\documentclass{elsarticle}
\usepackage{amsmath}
\usepackage{amsthm}
\usepackage{subfigure}
\usepackage{graphicx}
\usepackage{color}
\usepackage{pinlabel}
\usepackage{hyperref}
\usepackage{amsfonts}
\usepackage{amssymb}

\newlength\PullBackLength
\newcommand\PullBack[1][2]{%
  \setlength{\global\PullBackLength}{#1em}%
  \kern\PullBackLength%
  &
  \kern-\PullBackLength}

\newcommand{\ba}{\setminus}
\newcommand{\btu}{\bigtriangleup}
\newcommand{\btus}{\underline{\bigtriangleup}}
\newcommand{\F}{\mathcal{F}}

\newtheorem{theorem}{Theorem}[section]
\newtheorem{sublemma}{}[theorem]
\newtheorem{proposition}[theorem]{Proposition}
\newtheorem{lemma}[theorem]{Lemma}
\newtheorem{corollary}[theorem]{Corollary}

\theoremstyle{definition}
\newtheorem{definition}[theorem]{Definition}
\newtheorem{axiom}[theorem]{Axiom}

\theoremstyle{remark}
\newtheorem{remark}[theorem]{Remark}
\newtheorem{example}[theorem]{Example}

\numberwithin{equation}{section}

\DeclareMathOperator{\cell}{cell}

\newcommand{\spn}{w}

\begin{document}

\title{Matroids, delta-matroids and embedded graphs}
\author[cc]{Carolyn Chun}
\ead{chun@usna.edu}
\author[im]{Iain Moffatt}
\ead{iain.moffatt@rhul.ac.uk}
\author[sn]{Steven D. Noble\corref{cor1}}
\ead{steven.noble@brunel.ac.uk}
\author[rr]{Ralf Rueckriemen\fnref{fn1}}
\ead{ralf@rueckriemen.de}

\address[cc]{Mathematics Department, United States Naval Academy, Chauvenet Hall, 572C Holloway Road, Annapolis, Maryland 21402-5002, United States of America}
\address[im]{Department of Mathematics, Royal Holloway University of London, Egham, Surrey, TW20 0EX, United Kingdom}
\address[sn]{Department of Mathematics, Brunel University, Uxbridge, Middlesex, UB8 3PH, United Kingdom}
\address[rr]{   Aschaffenburger Strasse 23,  10779, Berlin}
\fntext[fn1]{Ralf Rueckriemen was financed by the DFG through grant RU 1731/1-1.}

\date{\today}

\begin{abstract}
Matroid theory is often thought of as a generalization of graph theory.
In this paper we propose an analogous correspondence between embedded graphs and delta-matroids.
We show that delta-matroids arise as the natural extension of graphic matroids to the setting of embedded graphs. We show that various basic ribbon graph operations and concepts have delta-matroid analogues, and  illustrate how the connections between embedded graphs and delta-matroids can be exploited.
Also, in direct analogy with the fact that the Tutte polynomial is matroidal, we show that several polynomials of embedded graphs from the literature, including the  Las Vergnas, Bollab\'as-Riordan and Krushkal polynomials, are in fact delta-matroidal. \end{abstract}


\maketitle

\section{Overview}

Matroid theory is often thought of as a generalization of graph theory. Many results in graph theory turn out to be special cases of results in matroid theory. This is beneficial in two ways.

First, graph theory can serve as an excellent guide for studying matroids.
As reported by Oxley, in~\cite{Oxley01}, Tutte famously observed that, ``If a theorem about graphs can be expressed in terms of edges and circuits alone it probably exemplifies a more general theorem about matroids.''
Perhaps one of the most spectacular illustrations of the effect of graph theory on matroid theory can be found in Geelen, Gerards and Whittle's recent and at the time of writing unpublished result that, for any finite field, the class of matroids that are representable over that field is well-quasi-ordered by the minor relation. This profound result is the matroid analogue of an equally profound result that came out of Robertson and Seymour's Graph Minors Project, in which, they proved that graphs are well-quasi-ordered by the minor relation~\cite{RS20}. Rather than the result itself, here  we want to focus on the fact that, to quote a recent statement of Whittle~\cite{whittle} about his work with Geelen and Gerards,
``It would be inconceivable to prove a structure theorem for matroids without the Graph Minors Structure Theorem as a guide''.

Second, insights from matroid theory can lead to new results about graphs. For example, Wu~\cite{Wu} established an upper bound for the number of edges of a loopless 2-connected graph, which was an improvement on existing results suggested by matroid duality.
Graph theory and matroid theory are mutually enriching, and this is the subject of~\cite{Oxley01} by Oxley.

The key purpose of this paper is to propose and study a similar correspondence between embedded graphs and delta-matroids.

Delta-matroids, introduced by Bouchet~\cite{ab1}, can be seen as a generalization of matroids. Where a matroid has bases, a delta-matroid has feasible sets.  These satisfy a symmetric exchange axiom, but do not all have to be of the same size.  We give a formal definition in the next section. The greater generality of delta-matroids allows us to capture not only information about a graph, but also about its embedding in a surface.  Bouchet was the first to observe a connection between embedded graphs and delta-matroids in~\cite{ab2}.  Our  approach is more direct than  his and has the advantage that it enables us to exploit the theory of ribbon graphs, much of which has developed since Bouchet did his work.

We will describe embedded graphs as ribbon graphs. The cycle matroid of a connected graph is constructed by taking the collection of spanning trees of the graph as its bases. In a connected ribbon graph, the spanning-trees are precisely the genus-zero spanning ribbon subgraphs that have exactly one boundary component. In the context of ribbon graphs, the genus-zero restriction is artificial, and it is subgraphs with exactly one boundary component, called \emph{quasi-trees} that play the role of trees. It turns out that the edge set of a ribbon graph together with its  spanning quasi-trees form a delta-matroid.

Moreover, we will see that this delta-matroid arises as the natural extension of a cycle matroid to the setting of embedded graphs, and that the delta-matroid structure follows from basic properties of surfaces.
We show that various concepts related to cellularly embedded graphs are special cases of concepts for delta-matroids. Because of this compatibility between the two structures, we extend Bouchet's initial ideas and propose that there is a correspondence between embedded graphs and delta-matroids that is analogous to the one between graphs and matroids.  We  justify this proposition  by illustrating how results from topological graph theory can be used to guide the development of delta-matroid theory, just as graph theory often guides matroid theory.  We also see that several polynomials of embedded graphs, including the Tutte, Las Vergnas, Bollob\'as-Riordan and Krushkal polynomials, are in fact delta-matroidal objects, just as many graph polynomials are matroidal.

The paper is structured as follows.
In Section~\ref{matroidsanddelta-matroids}, we give an overview of some relevant properties of matroids and delta-matroids.  Section~\ref{ribbongraphs} contains some background on cellularly embedded graphs. Most of the time, we will use the language of ribbon graphs instead of cellularly embedded graphs.  These are equivalent concepts (see Figure~\ref{f.desc}), but ribbon graphs have the advantage of being closed under the natural minor operations.

In Section~\ref{s4}, we describe how delta-matroids arise from ribbon graphs, emphasising that they arise as the natural extensions of various classes of matroids associated with graphs. We show that some of  these delta-matroids, albeit in a different language, appeared in Bouchet's foundational work in delta-matroids. In Section~\ref{s5} we discuss their connections with graphic matroids and describe how basic properties of a ribbon graph are encoded in its delta-matroid. We provide evidence of the basic compatibility between delta-matroids and ribbon graphs.   In particular, we prove that one of the most fundamental operations of delta-matroids, the twist, is the delta-matroid analogue of a partial dual of a ribbon graph, which turns out to be a key result in connecting the two areas. We describe how to see edge structure and connectivity in a ribbon graph in terms of its delta-matroid, and show how results on delta-matroid connectivity inform ribbon graph theory. We also demonstrate that excluded minor characterisations that have appeared in both the delta-matroid and ribbon graph literature are translations of one another.

In Section~\ref{s6}, we discuss various polynomials. Some well-known graph polynomials, and in particular the Tutte polynomial, are properly understood as matroid polynomials, rather than graph polynomials. There has been considerable recent interest in extensions of the Tutte polynomial to  graphs embedded in surfaces. Three generalizations of the Tutte polynomial to embedded graphs in the literature are the Las~Vergnas polynomial, the Bollob\'as-Riordan polynomial, and the Kruskal polynomial. We show that each of these generalizations is determined by the delta-matroids of ribbon graphs, and that the ribbon graph polynomials are special cases of more general  delta-matroid polynomials. That is, while the Tutte polynomial is properly a matroid polynomial, its topological extensions are properly delta-matroid polynomials.

Our results here offer new perspectives on delta-matroids. We illustrate here a fundamental interplay between ribbon graphs and delta-matroids, that is analogous to the interplay between graphs and matroids. By doing so we offer a new approach to delta-matroid theory.

\section{Matroids and delta-matroids}
\label{matroidsanddelta-matroids}
Our terminology follows~\cite{ab1} and~\cite{Oxley11}, except where explicitly stated.

\subsection{Set systems and delta-matroids}
A \emph{set system} is a pair $D=(E,{\mathcal{F}})$ where $E$ is a set, which we call the \emph{ground set}, and $\mathcal{F}$ is a collection of subsets of $E$. The members of $\mathcal{F}$ are called \emph{feasible sets}. A set system is \emph{proper} if $\mathcal{F}$ is not empty; it is \emph{trivial} if $E$ is empty.
For a set system $D$ we will often use $E(D)$ to denote its ground set and $\mathcal{F}(D)$ its collection of feasible sets. In this paper we will always assume that $E$ is a finite set and will do so without further comment.

The \emph{symmetric difference} of sets $X$ and $Y$, denoted by $X\bigtriangleup Y$,  is  $(X\cup Y)-(X\cap Y)$.

A \emph{delta-matroid} is a proper set system $D=(E,{\mathcal{F}})$ that satisfies the Symmetric Exchange Axiom:
\begin{axiom}[Symmetric Exchange Axiom]
\label{sea}
 For all $(X,Y,u)$ with $X,Y\in \mathcal{F}$ and $u\in X\bigtriangleup Y$,  there is an element $v\in X\bigtriangleup Y$ such that $X\bigtriangleup \{u,v\}$ is in $\mathcal{F}$.
\end{axiom}

Note that we allow $v=u$ in the Symmetric Exchange Axiom.

If the feasible sets of a delta-matroid are equicardinal, then the delta-matroid is a \emph{matroid} and we refer to its feasible sets as its \emph{bases}.
If a set system forms a matroid $M$, then we usually denote $M$ by $(E,\mathcal{B})$,  and often use $\mathcal{B}(M)$ to denote its collection of bases $\mathcal{B}$. It is not hard to see that the definition of a matroid given here is equivalent to the `usual'  definition of a matroid through bases given in, for example,~\cite{Oxley11,Webook}.

Throughout this paper, we will often omit the set brackets in the case of a single element set.
For example, we write $E-e$ instead of $E-\{e\}$, or $F\cup e$ instead of $F\cup \{e\}$.

\subsection{Graphic matroids}
For a graph $G=(V,E)$ with $k$ connected components, let $\mathcal{B}$ be the edge sets of the maximal spanning forests of $G$.
$\mathcal{B}$ is obviously non-empty, and its elements are equicardinal  since each spanning forest of $G$ has $|V|-k$ edges. It is not too hard to see that the Symmetric Exchange Axiom holds, and so
the set system $M(G)=(E,\mathcal{B})$ is a  matroid, which is called the \emph{cycle matroid of $G$}.
Any matroid that is the cycle matroid of a graph is a \emph{graphic matroid}.

\begin{example}\label{e.mundane1}
If $G$ is the graph shown in Figure~\ref{f.desca}, then $M(G)=(E,\mathcal{B})$ where $E=\{1,2,3,4\}$ and $ \mathcal{B}= \{\{1\}, \{2\} \}$.
\end{example}

\subsection{Matroid rank}
Let $M$ be a matroid with ground set $E$.
 A subset $I$ of $E$ is an \emph{independent set} of $M$ if and only if it is a subset of a basis of $M$.
A \emph{rank function} is defined for all subsets of the ground set of a matroid. Its value on a subset $A$ of $E$ is the cardinality of the largest independent set contained in $A$.
The rank of a set $A$ is written $r_M(A)$, or just $r(A)$ if the matroid is clear from the context.
Thus,
$ r_M(A)=\text{max}\{|A\cap B| \mid B\in\mathcal{B}(M)\}$.
We say that the \emph{rank of $M$}, written $r(M)$, is equal to $r(E)$, which is equal to $|B|$, for any $B\in\mathcal{B}(M)$.

\begin{example}\label{e.mundane2}
For a graph $G=(V,E)$, the  rank function of its cycle matroid $M=M(G)$ is given by $r(A)=|V|-k(A)$, where $k(A)$ is the number of connected components of the spanning subgraph $(V,A)$ of $G$, and $A\subseteq E$.
\end{example}

\subsection{Width and evenness}
For a delta-matroid $D=(E,\mathcal{F})$, let $\mathcal{F}_{\max}(D)$ and $\mathcal{F}_{\min}(D)$ be the set of feasible sets with maximum and minimum cardinality, respectively. We will usually omit $D$ when the context is clear.
Let $D_{\max}:=(E,\mathcal{F}_{\max})$ and let $D_{\min}:=(E,\mathcal{F}_{\min})$.
Then $D_{\max}$ is the \emph{upper matroid} and $D_{\min}$ is the \emph{lower matroid} of $D$. These matroids were defined by Bouchet in~\cite{ab2}.
It is straightforward to show that the upper matroid and the lower matroid are indeed matroids.
The \emph{width of $D$}, denoted by $\spn(D)$, is defined by  \[\spn(D):=r(D_{\max})-r(D_{\min}).\]
Thus the width of $D$ is the difference between the sizes of its largest and smallest feasible sets.

If the sizes of the feasible sets of a delta-matroid all have the same parity, then we say that the delta-matroid is \emph{even}.
Otherwise, we say that the delta-matroid is \emph{odd}.
In particular, every matroid is an even delta-matroid.
 It is perhaps worth emphasising that an even delta-matroid need not have feasible sets of even cardinality.

It is convenient to record the following useful result here.
\begin{lemma}\label{lem:useful}
Let $D=(E,\mathcal{F})$ be a delta-matroid, let $A$ be a subset of $E$ and let $s_0 = \min\{|B \cap A| \mid B\in\mathcal {B}(D_{\min})\}$. Then for any $F\in\mathcal{F}$ we have $|F\cap A|\geq s_0$.
\end{lemma}
\begin{proof}
We proceed by contradiction. If $s_0=0$, then there is nothing to prove, so we can assume that $s_0>0$. Suppose that $F\in\mathcal F$ and $|F\cap A|<s_0$. Choose $F'\in \mathcal {F}_{\min}$ with $|F'\cap A|=s_0$ and $|F'\cap F\cap A|$ as large as possible. Now there exists $x\in A\cap(F'- F)$ and so $x\in F' \bigtriangleup F$. Hence there exists $y$ belonging to $F' \bigtriangleup F$ such that $F''=F'\bigtriangleup \{x,y\} \in \mathcal F$. Because $F'\in\mathcal {F}_{\min}$, we have $y\in F-F'$. And because $|F' \cap A|=s_0$, we must have $y\in F\cap A$. But then $F''\in\mathcal {F}_{\min}$, $|F''\cap A|=s_0$ and $|F''\cap F\cap A| > |F'\cap F\cap A|$, contradicting the choice of $F'$.
\end{proof}

\subsection{Twists, duals, loops, coloops, and minors}
Twists, introduced by Bouchet in~\cite{ab1}, are one of the fundamental operations of delta-matroid theory.
\begin{definition}
Let $D=(E,{\mathcal{F}})$ be a set system.
For $A\subseteq E$, the  \emph{twist} of $D$ with respect to $A$, denoted by $D* A$, is given by $(E,\{A\bigtriangleup X \mid  X\in \mathcal{F}\})$.
The \emph{dual} of $D$, written $D^*$, is equal to $D*E$.
\end{definition}

It follows easily from the identity $(F'_1\btu A)\btu(F'_2\btu A)=F'_1\btu F'_2$ that the twist of a delta-matroid is also a delta-matroid.
We restate this fact in the following lemma.
\begin{lemma}[Bouchet~\cite{ab1}]
\label{obvious}
Let $D$ be a delta-matroid and let $A$ be a subset of $E(D)$.
Then $D*A$ is a delta-matroid.
\end{lemma}

Although it is always a delta-matroid, a twist of a matroid $M=(E,\mathcal{B})$ need not be a matroid. (For example, if $M=(\{1,2\}, \{\{1\},\{2\}\}))$ then $M\ast\{1\}$ has feasible sets $\{\emptyset, \{1,2\}\}$ and so is not a matroid.)
However, its dual $M^*=M\ast E$ is always a matroid.
The rank function of $M^*$ is given by
\begin{equation}\label{eq:matrank}
r_{M^*}(A)=r_M(E-A)+|A|-r_M(E).
\end{equation}

For a delta-matroid $D=(E,\mathcal{F})$, and $e\in E$, if $e$ is in every feasible set of $D$, then we say that $e$ is a \emph{coloop of $D$}.
If $e$ is in no feasible set of $D$, then we say that $e$ is a \emph{loop of $D$}.
Note that a coloop or loop of $D$ is a loop or coloop, respectively, of $D*A$ for any subset $A$ of $E$ containing $e$.

 If $e$ is not a coloop, then, following Bouchet and Duchamp~\cite{BD91}, we define $D$ \emph{delete} $e$, written $D\ba e$, to be
 \[D\ba e:=(E-e, \{F \mid F\in \mathcal{F}\text{ and } F\subseteq E-e\}).\]
 If $e$ is not a loop, then we define $D$ \emph{contract} $e$, written $D/e$, to be \[D/e:=(E-e, \{F-e \mid F\in \mathcal{F}\text{ and } e\in F\}).\]
 If $e$ is a
loop or a coloop, then one of $D\ba e$ and $D/ e$ has already been
defined, so we can set $D/e=D\ba e$.

Both $D\ba e$ and $D/e$ are delta-matroids (see~\cite{BD91}).
Let $D'$ be a delta-matroid obtained from $D$ by a sequence of deletions and contractions.  Then $D'$ is independent of the order of the  deletions and contractions used in its construction (see~\cite{BD91}) and $D'$ is called a \emph{minor} of $D$.
If $D'$ is formed from $D$ by deleting the elements of $X$ and contracting the elements of $Y$ then we write
$D'=D\setminus X/Y$. The \emph{restriction} of $D$ to a subset $A$ of $E$, written $D|A$, is equal to $D\ba (E-A)$.

Note that $D^*\ba e=(D/e)^*$. The next result shows that deletion, contraction and twists are also related. It is a reformulation of  Property~2.1 of~\cite{BD91}.
\begin{lemma}\label{le:delcondual1}
For a delta-matroid $D$ and distinct elements $e$ and $f$ of $E(D)$, we have
\begin{enumerate}
\item  $D\setminus e = ((D*f) \setminus e)*f$ and $D/e = ((D*f) / e)*f$;
\item $D\setminus e = (D*e) / e$ and $D/e = (D*e) \setminus e$.
\end{enumerate}
\end{lemma}

Using Lemma~\ref{le:delcondual1}  and induction we obtain the following.
\begin{proposition}\label{pr:delcondualmany}
Let $D$ be a delta-matroid and let $A,X$, and $Y$ be subsets of $E(D)$ with $X\cap Y=\emptyset$. Then
\[(D*A)\setminus X/Y = (D \ba ((X-A)\cup (Y\cap A))  / ((Y-A)\cup (X\cap A))) * (A-X-Y).\]
In particular, $D\ba X=(D^*/X)^*$ and, when $A$ is the disjoint union of $X$ and $Y$, we have
\[ (D*A)\ba X/Y = D \ba Y  / X . \]
\end{proposition}

\subsection{Delta-matroid rank}
Bouchet defined an analogue of the rank function for delta-matroids in~\cite{abrep}. For a delta-matroid $D=(E,\mathcal{F})$, it is denoted by $\rho_D$ or simply $\rho$ when $D$ is clear from the context. Its value on a subset $A$ of $E$ is given by \[\rho(A):=|E|-\min\{|A\bigtriangleup F| \mid F\in \mathcal{F}\}.\]
Note that the feasible sets of a delta-matroid can be recovered from its rank function.

An easy consequence of basic properties of the symmetric difference operation is the following.
\begin{lemma} \label{le:rankdual}
Let $D$ be a delta-matroid and let $A$ be a subset of $E(D)$. Then $\rho_{D^*}(A)= \rho_D(E-A)$.
\end{lemma}

The next two results show how the rank function changes when an element is deleted or contracted.
\begin{lemma}
\label{avoide}
Let $D=(E,\mathcal{F})$ be a delta-matroid and let $e$ be an element in $E$, and $X$ a subset of $E-e$.
Then either $e$ is a coloop or there exists $F\in\mathcal{F}$ such that $\rho (X)=|E|-|X\btu F|$ and $e\notin F$.
\end{lemma}
\begin{proof}
Suppose $e$ is not a coloop.
Then there is a feasible set $F$ avoiding $e$.
Take $F'\in\mathcal{F}$ such that $\rho(X)=|E|-|X\btu F'|$.
If $F'$ avoids $e$ then the lemma holds, so we assume this is not the case.
Then $e\in F'\btu F$, so the Symmetric Exchange Axiom (Axiom~\ref{sea}) implies that there exists $f\in F'\btu F$ such that $F''=F'\btu \{e,f\}\in\mathcal F$.
If $f=e$, then $|X\btu F''| = |X\btu (F'-e)|<|X \btu F'|$ which is not possible, because $\rho(X)=|E|-|X\btu F'|$. So $f\ne e$ and $X\btu F''=X\btu (F'\btu \{e,f\})=X\btu ((F'-e)\btu f)=(X\btu (F'-e))\btu f$, so we deduce that  $|X\btu F''|\leq|X\btu F'|$.
As $F'$ was chosen from $\mathcal{F}$ to minimize $|X\btu F'|$, we deduce that $|X\btu F''|=|X\btu F'|$.
Since $e\notin F''$, the lemma holds.
\end{proof}

\begin{lemma}
\label{minorbirank}
Let $D=(E,\mathcal{F})$ be a delta-matroid and let $e$ be an element in $E$, and $X$ a subset of $E-e$. Then

\begin{align}\label{eq:rankdel}\rho _{D\ba e}(X) &= \begin{cases} \rho _D(X), & \mbox{if }e \mbox{ is a coloop of } D\\ \rho _D(X) -1, & \mbox{otherwise} \end{cases}\\
\intertext{and}
\label{eq:rankcon}\rho_{D/e}(X) &= \begin{cases} \rho_D(X \cup e), & \mbox{if }e \mbox{ is a loop of } D\\ \rho_D(X\cup e) -1, & \mbox{otherwise.} \end{cases}\end{align}
\end{lemma}
\begin{proof}We first establish~\eqref{eq:rankdel}.
Suppose that $e$ is not a coloop. Lemma~\ref{avoide} implies that there exists $F\in\mathcal{F}(D)$ such that $e\notin F$ and $\rho_D(X)=|E|-|X\btu F|$.
Thus $\rho _{D\ba e}(X) \leq |E-e|-|X\btu F|=|E|-|X\btu F|-1=\rho_D(X)-1$.
Moreover every feasible set of $D\setminus e$ is a feasible set of $D$. Hence $\rho_D(X) \leq \rho_{D\ba e}(X)+1$. Combining these two inequalities gives the result.

Suppose that $e$ is a coloop of $D$. Let $A$ be a subset of $E-e$. Then $A$ is a feasible  set of $D\setminus e$ if and only if $A\cup e$ is a feasible set of $D$. Furthermore, $|X\btu A|=|X\btu (A\cup e)|-1$.
Take $F\in\mathcal{F}(D \ba e)$ such that $\rho _{D\ba e}(X)=|E|-|X\btu F|$. Then $F\cup e$ is in $\mathcal{F}$ and has smallest symmetric difference with $X$ of all feasible sets in $\mathcal{F}$.  Thus $\rho _{D}(X)=|E|-|X\btu (F\cup e)|=|E|-|X\btu F|-1=|E-e|-|X\btu F|=\rho _{D\ba e}(X)$.

Now Equation~\eqref{eq:rankcon} is obtained by using duality.
Lemma~\ref{le:rankdual} implies that $\rho_{D/e}(X)=\rho_{(D/e)^*}(E-e-X)=\rho_{D^*\ba e}(E-e-X)$. Using Equation~\eqref{eq:rankdel}, we obtain
\[\rho_{D^*\ba e}(E-e-X) = \begin{cases} \rho_{D^*}(E-e-X), & \mbox{if }e \mbox{ is a coloop of } D^*\\ \rho_{D^*}(E-e-X) -1, & \mbox{otherwise.} \end{cases}\]
The result follows by applying duality again and noting that $e$ is a coloop of $D^*$ if and only if it is a loop of $D$.
\end{proof}


\section{Ribbon graphs}
\label{ribbongraphs}
We are concerned here with connections between cellularly embedded graphs and delta-matroids. As it is much more convenient for our purposes,  we realize cellularly embedded graphs as ribbon graphs. This section provides a  brief overview of ribbon graphs, as well as standard ribbon graph notation and constructions.
A more thorough treatment of the topics covered in this section can be found in, for example,~\cite{EMMbook}.

\subsection{Cellularly embedded graphs and ribbon graphs}
\subsubsection{Ribbon graphs}\label{sss.ribbongraph}
A \emph{cellularly embedded  graph} $G\subset \Sigma$ is a graph drawn on a closed compact surface $\Sigma$  in such a way that edges only intersect at their ends,  and such that each connected component of $\Sigma - G$
is homeomorphic to a disc.
Note that each connected component of $G$ must be embedded in a different component of the surface.

Two cellularly embedded graphs $G\subset \Sigma$ and  $G'\subset \Sigma'$  are
 \emph{equivalent} if there is a homeomorphism,
 $\varphi:\Sigma \rightarrow \Sigma'$,
 which is orientation-preserving if $\Sigma$ is orientable, and has the property that $\varphi|_{G}:G\rightarrow G'$ is a graph isomorphism. We consider cellularly embedded graphs up to equivalence.

Ribbon graphs provide an alternative, and more natural for the present setting, description of cellularly embedded graphs.
\begin{definition}
A \emph{ribbon graph} $G =\left(  V(G),E(G)  \right)$ is a surface with boundary, represented as the union of two  sets of  discs: a set $V (G)$ of \emph{vertices} and a set of \emph{edges} $E (G)$ with the following properties.
\begin{enumerate}
 \item The vertices and edges intersect in disjoint line segments.
 \item Each such line segment lies on the boundary of precisely one vertex and precisely one edge. In particular, no two vertices intersect, and no two edges intersect.
 \item Every edge contains exactly two such line segments.
\end{enumerate}
\end{definition}

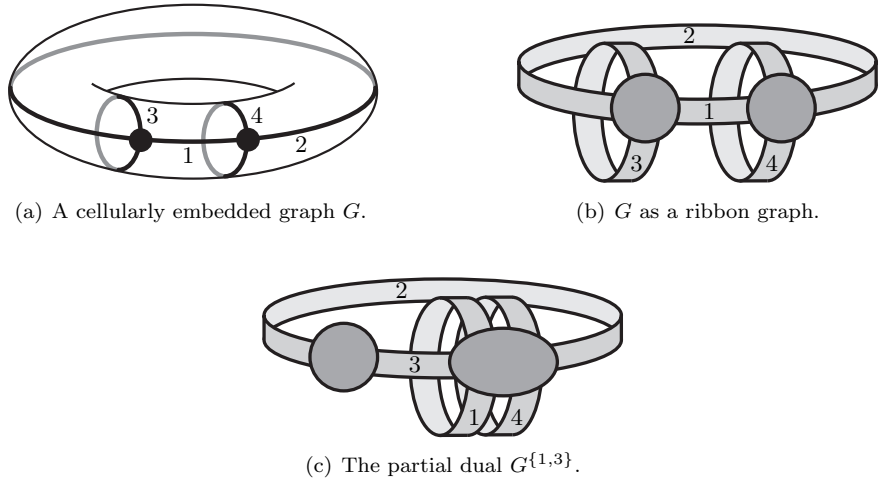
\begin{figure}
\centering
\subfigure[A cellularly embedded graph $G$.]{
\labellist \small\hair 2pt
\pinlabel {$1$} at   84 14
\pinlabel {$2$} at    135 19
\pinlabel {$3$} at    68 32
\pinlabel {$4$} at   115 32
\endlabellist
\raisebox{0mm}{\includegraphics[height=2.5cm]{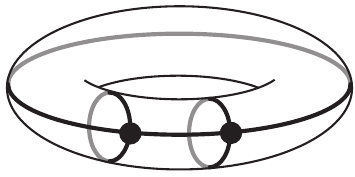}}
 \label{f.desca}
}
\qquad\qquad
\subfigure[$G$ as a ribbon graph.]{
\labellist \small\hair 2pt
\pinlabel {$1$} at   60 22.7
\pinlabel {$2$} at   54 45.6
\pinlabel {$3$} at   38 7
\pinlabel {$4$} at   79 7
\endlabellist
\includegraphics[height=2.5cm]{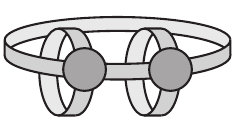}
 \label{f.descb}
}
\subfigure[The partial dual $G^{\{1,3\}}$.]{
\labellist \small\hair 2pt
\pinlabel {$1$} at   66 7
\pinlabel {$2$} at   44 45.3
\pinlabel {$3$} at   48 23
\pinlabel {$4$} at   79 7
\endlabellist
\includegraphics[height=2.5cm]{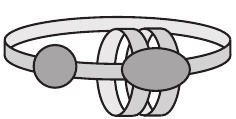}
 \label{f.descc}
}
\caption{Embedded graphs and ribbon graphs.}
\label{f.desc}
\end{figure}

It is well-known that ribbon graphs are just descriptions of  cellularly embedded
graphs (see for example~\cite{GT87}).  If $G$ is a cellularly embedded graph, then a ribbon
graph representation results from taking a small neighbourhood  of
the cellularly embedded graph $G$, and deleting its complement. On the other hand, if $G$ is a
ribbon graph, then, topologically, it is a  surface with boundary. Capping off the holes, that is, `filling in' each hole by identifying its boundary component with the boundary of a disc, results in a ribbon graph embedded in a closed surface from which  a graph embedded in the surface is readily obtained.  Figure~\ref{f.desc} shows an embedded graph described as both a  cellularly embedded graph and a ribbon graph. We say that two ribbon graphs are \emph{equivalent} if they define equivalent cellularly embedded graphs, and we consider ribbon graphs up to equivalence. This means that ribbon graphs are considered up to homeomorphisms that preserve the graph structure  of the ribbon graph and the cyclic order of half-edges at each of its vertices.

\subsubsection{Ribbon subgraphs and edge deletion}\label{sss.srs}
Let $G=(V,E)$ be a ribbon graph. Then a ribbon graph $H$ is a \emph{ribbon subgraph}  of $G$ if it can be obtained by removing vertices and edges of $G$. If $V(H)=V(G)$ then $H$ is a \emph{spanning ribbon subgraph} of $G$. Note that every subset $A$ of $E$ uniquely determines a spanning ribbon subgraph $(V,A)$ of $G$.

If $e$ is an edge of $G$, then $G$ \emph{delete} $e$, written $G\ba e$,  is defined to be the ribbon subgraph $(V, E-e)$ of $G$. Similarly, for $A\subseteq E$,  $G\ba A$ is defined to be $(V, E-A)$.  Table~\ref{tablecontractrg} shows the local effect of deleting an edge of a ribbon graph.

An important observation about ribbon subgraphs is that if a ribbon graph $G$ is realised as a graph cellularly embedded in a surface $\Sigma$,  and $G\ba e$, or a ribbon subgraph $H$ of $G$, is realised  as a  graph cellularly embedded in a surface $\Sigma'$, then $\Sigma$ and $\Sigma'$ need not be homeomorphic.

\subsubsection{Standard parameters}\label{sss.st}
A ribbon graph is a graph with additional structure and so standard graph terminology carries over to ribbon graphs.
If $G$ is a ribbon graph, then  $v(G)$ and $e(G)$ denote $|V(G)|$ and $|E(G)|$, respectively.  Furthermore, $k(G)$ denotes the number of connected components in $G$, and $f(G)$ is the number of boundary components of the surface defining the ribbon graph.
For example, the ribbon graph $G$ of Figure~\ref{f.descb} has $f(G)=2$.
Note that, if $G$ is realised as a cellularly embedded graph, then $f(G)$ is the number of its faces.
The \emph{rank} of $G$, denoted by $r(G)$, is defined to be $v(G)-k(G)$, and the \emph{nullity} of $G$, denoted by $n(G)$, is defined to be $e(G)-r(G)$.

A ribbon graph $G$ is \emph{orientable} if it is orientable when regarded as a surface. We define a ribbon graph parameter $t$ by setting $t(G)=1$ if $G$ is non-orientable, and $t(G)=0$ otherwise.

The \emph{genus} of a ribbon graph $G$ is its genus when regarded as a surface. If $G$ is realized as a graph cellularly embedded in $\Sigma$, then its genus is exactly the genus of $\Sigma$, and $G$ is orientable if and only if $\Sigma$ is.
  The \emph{Euler genus}, $\gamma(G)$, of  $G$ is the genus of $G$ if $G$ is non-orientable, and is twice its genus if $G$ is orientable.  Euler's formula gives $\gamma(G)=2k(G)-v(G)+e(G)-f(G)$.
  We say that a ribbon graph $G$ is   \emph{plane}  if $\gamma(G)=0$. Note that we allow plane graphs to have more than one connected component. Plane ribbon graphs correspond to graphs that can be cellularly embedded in some disjoint union of spheres.

For each subset $A$ of $E$, we let $r(A)$, $k(A)$, $n(A)$,  $f(A)$, $t(A)$, and $\gamma(A)$ each refer to the spanning ribbon subgraph $(V,A)$ of $G$, where $G$ is given by context. When the choice of $G$ is not clear from the context, we write $r_G(A)$, $k_G(A)$, etc..
Observe that the function $r$ on $E$ defined here coincides with the rank function of the cycle matroid $M(G)$ of $G$.

The following result is an obvious, but useful, consequence of the fact that each edge of a ribbon graph  meets one or two boundary components.
\begin{proposition}\label{obv}
If $G$ is ribbon graph, $A\subseteq E(G)$ and $e\in E(G)$, then $f(A)$ and $f(A\btu e)$ differ by at most one.
\end{proposition}

\subsubsection{Loops and bridges}\label{sss.lb}
An edge $e$ of a ribbon graph $G$ is a \emph{bridge} if $k(G\ba e)>k(G)$.
The edge $e$ is a \emph{loop} if it is incident with exactly one vertex.
We will abuse notation and also use the term loop to describe the ribbon subgraph of $G$ consisting of $e$ and its incident vertex.
In ribbon graphs, loops can have various properties. A loop or cycle is said to be \emph{non-orientable} if it is homeomorphic to a M\"obius band. Otherwise it is \emph{orientable}.
 Two cycles $C_1$ and $C_2$ in $G$ are said to be \emph{interlaced} if there is  a vertex $v$   such that   $V(C_1)\cap V(C_2)=\{v\}$, and  $C_1$ and $C_2$ are met in the cyclic order $C_1\,C_2\,C_1\,C_2$ when travelling around the boundary of the vertex $v$. A loop is \emph{non-trivial} if it is interlaced with  some cycle in $G$, otherwise it is \emph{trivial}.

\subsubsection{Ribbon graph minors}\label{sss.rgmin}
For a ribbon graph $G$ with an edge $e$ recall that $G\ba e$ is obtained by removing $e$ from $G$.  Similarly, if  $v$ is a vertex of $G$, then the \emph{vertex deletion} $G \ba v$ is defined to be the ribbon graph obtained from $G$ by removing the vertex $v$ together with all its incident edges.

The definition of edge contraction, introduced in~\cite{BR2,Ch1}, is a little more involved than that of edge deletion.
\begin{definition}\label{d.cont}
 Let $G$ be a ribbon graph. Let $e\in E(G)$ and  $u$ and $v$ be its incident vertices, which are not necessarily distinct. Then  $G/e$ denotes the ribbon graph obtained as follows. Consider the boundary component(s) of $e\cup u \cup v$ as curves on $G$. For each resulting curve, attach a disc, which will form a vertex of $G/e$, by identifying its boundary component with the curve. Delete $e$, $u$ and $v$ from the resulting complex.
We say that $G/e$ is obtained from $G$ by \emph{contracting} $e$.
\end{definition}

A ribbon graph $H$ is a \emph{minor} of a ribbon graph $G$ if $H$ is obtained from $G$ by a sequence of edge deletions, vertex deletions, and edge contractions.

The local effect of contracting an edge of a ribbon graph is shown in Table~\ref{tablecontractrg}. Observe that contracting an edge may change the number of vertices  or  orientability of a ribbon graph.
Since deletion and contraction are local operations, if some edges in a ribbon graph are deleted and some others are contracted, then the same ribbon graph will be produced regardless of the order of operations.

 The definition of edge contraction might be a little surprising at first.
However, the reader should see that it is  natural upon observing that  Definition~\ref{d.cont} is just an expression of the
obvious idea of contraction  as the `identification of $e$ and its incident vertices into a single vertex' in a way that allows it to be applied to loops. (See also the discussion in~\cite{EMMbook} on this topic.)
Unlike for graphs, when working with ribbon graph minors it is necessary to be able to contract loops as otherwise the set of ribbon graphs will contain infinite anti-chains when quasi-ordered using the minor relation (see~\cite{Moprep}).

\begin{table}
\centering
\begin{tabular}{|c||c|c|c|}\hline
 &  non-loop & non-orientable loop&orientable loop\\ \hline
\raisebox{6mm}{$G$} &
\includegraphics[scale=.25]{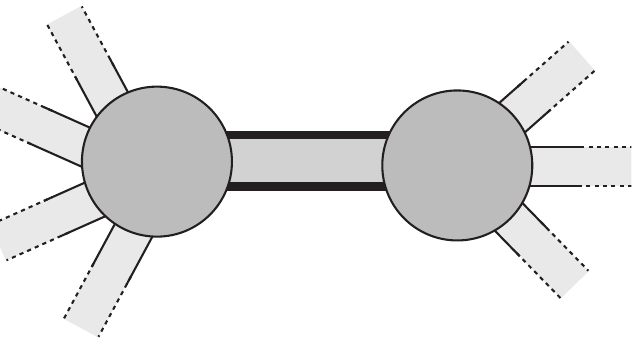} &\includegraphics[scale=.25]{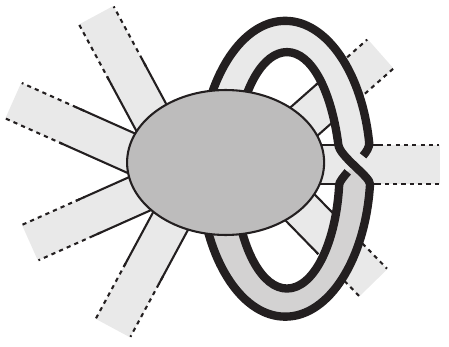} &\includegraphics[scale=.25]{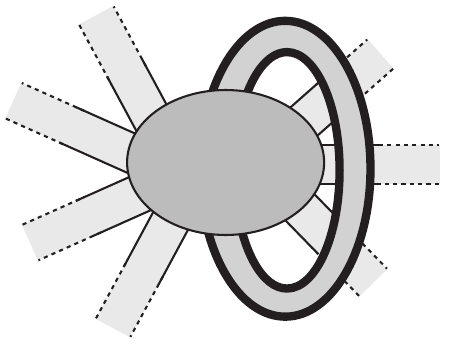}
\\ \hline
\raisebox{6mm}{$G\ba e$}
&
\includegraphics[scale=.25]{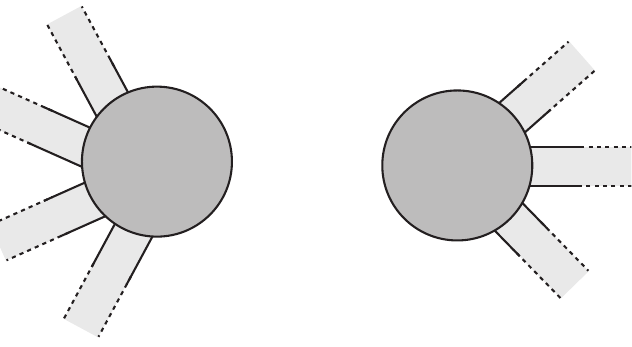} &\includegraphics[scale=.25]{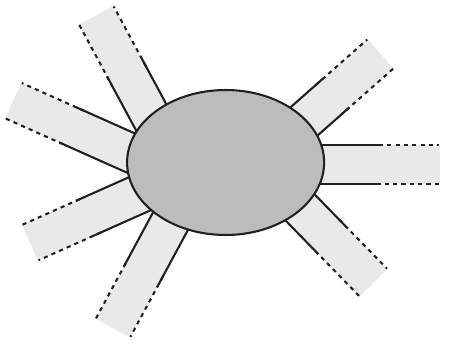}&\includegraphics[scale=.25]{ch4_35a}
\\ \hline
\raisebox{6mm}{\begin{tabular}{l} $G/e$ \\ $=G^{e}\setminus e$\end{tabular}}
&
\includegraphics[scale=.25]{ch4_35a} &\includegraphics[scale=.25]{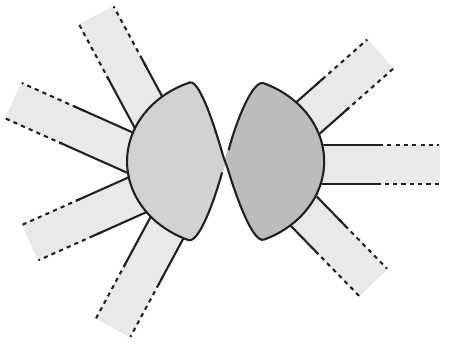}&\includegraphics[scale=.25]{ch4_38a} \\ \hline
\raisebox{6mm}{$G^{e}$} &
\includegraphics[scale=.25]{ch4_35} &\includegraphics[scale=.25]{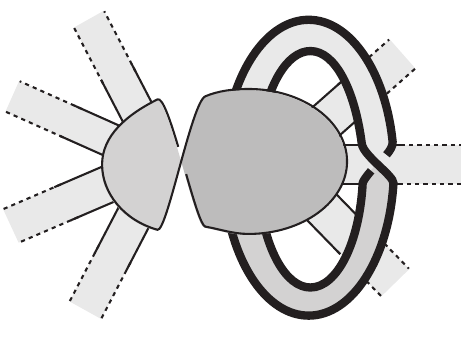} &\includegraphics[scale=.25]{ch4_38}
\\ \hline
\end{tabular}
\caption{Operations on an edge $e$ (highlighted in bold) of a ribbon graph. The ribbon graphs are identical outside of the region shown.}
\label{tablecontractrg}
\end{table}

\subsubsection{Separability}\label{sss.se}
For a ribbon graph $G$ and non-trivial ribbon subgraphs $P$ and $Q$ of $G$, we write $G=P\sqcup Q$ when $G$ is the \emph{disjoint union} of $P$ and $Q$, that is,  when $G=P\cup Q$ and $P\cap Q=\emptyset$.
A vertex $v$ of $G$ is  a \emph{separating vertex} if there are non-trivial ribbon subgraphs $P$ and $Q$ of $G$ such that $G=P\cup Q$ and $P\cap Q=\{v\}$. In this case we write $G=P\oplus Q$.

We write $G=P\curlyvee Q$, if $G=P \oplus Q$ and no cycle in $P$ is interlaced with a cycle in $Q$.
 Observe  it is possible that $G=P\curlyvee Q$ and  $G'=P\curlyvee Q$, for non-equivalent ribbon graphs $G$ and $G'$.

(We remark that here there is a close relationship with the join operation, $\vee$, on ribbon graphs:  $G=P\curlyvee Q$ if and only if
$P=G_1 \vee \cdots \vee G_i$,  $Q=G_{i+1} \vee \cdots \vee G_n$, and, for some permutation $\sigma$,
$G=G_{\sigma(1)} \vee \cdots \vee G_{\sigma(n)}$,  where each join occurs at the same vertex.  We refer the reader to~\cite{Mo5,Mof11c} for a fuller discussion of separability for ribbon graphs.)

\subsection{Geometric duals and partial duals}\label{ss.dual}
The construction of the \emph{geometric dual}, $G^*$, of a cellularly embedded graph $G$ is well known: $V(G^*)$ is obtained by placing one vertex in each face of $G$, and $E(G^*)$ is obtained by embedding an edge of $G^*$ between two vertices whenever the faces of $G$ in which they lie are adjacent. Geometric duality has a particularly neat description when translated to the  language of ribbon graphs. Let $G=(V(G), E(G))$ be a ribbon graph. Recalling that, topologically, a ribbon graph is a  surface with boundary, we cap off the holes using a set of discs, denoted by $V(G^*)$, to obtain a surface without boundary. The \emph{geometric dual} of $G$  is the ribbon
graph $G^* = (V(G^*), E(G))$. Observe that, for ribbon graphs, the edges of $G$  and $G^*$ are identical.  The only change is which arcs on their boundaries do and do not intersect vertices. This allows us to consider a subset $A$ of edges of $G$ as also being a subset of edges of $G^*$ and vice versa. We adopt this convention. Although it is common to distinguish the two sets by writing $A$ and $A^*$, doing so proves to be notationally difficult in the current setting.

Chmutov, in~\cite{Ch1}, introduced a far-reaching generalization of geometric duality, called partial duality.
Roughly speaking, a partial dual of a ribbon graph  is obtained by forming the geometric dual with respect to only a subset of its edges. Partial duality arises as a natural operation in knot theory, topological graph theory, graph polynomials, and quantum field theory. We will see later that it is also an analogue of a fundamental operation on delta-matroids. Here we define partial duals directly on ribbon graphs.  We refer the reader to~\cite{Ch1,EMM,Mo4}  or the exposition~\cite{EMMbook} for alternative constructions and other perspectives of partial duals.

Let $G=(V,E)$ be a ribbon graph and $A\subseteq E$. The  partial dual $G^{A}$ of $G$ is obtained by forming the geometric dual of $G$ as described above but ignoring the edges not in $A$ as follows.  Regard the boundary components of the spanning ribbon subgraph $(V,A)$ of $G$ as curves on the surface of $G$. Glue a disc to $G$ along each connected component of this curve and remove the interior of all vertices of $G$. The resulting ribbon graph is the \emph{partial dual} $G^{A}$.

We identify the edges of $G$ with those of $G^A$ using the natural correspondence.
 Table~\ref{tablecontractrg} shows the local effect of partial duality on an edge $e$ (highlighted in bold) of a ribbon graph $G$. The ribbon graphs are identical outside of the regions shown. In fact Table~\ref{tablecontractrg} serves as a perfectly adequate definition of partial duality for this paper.

Observe from Table~\ref{tablecontractrg} that $e$ is a bridge of $G$ if and only if $e$ is a trivial orientable loop in $G^{e}$;  $e$ is a non-loop non-bridge edge of $G$ if and only if $e$ is a non-trivial orientable loop in $G^{e}$; and $e$ is a (non-)trivial non-orientable loop in $G$ if and only if $e$ is a (non-)trivial non-orientable   loop in $G^{e}$.
We also record the following basic properties of partial duality for  use later.
\begin{proposition}[Chmutov~\cite{Ch1}]\label{p.pd2}
Let $G$ be a  ribbon graph and $A, B\subseteq E(G)$.  Then
\begin{enumerate}
\item $G^{E(G)}=G^*$ and $G^{\emptyset}=G$;
\item $(G^A)^B=G^{A\triangle B}$;
\item \label{p.pd2.3}$G/e= G^{e}\ba e$;
\item $G$ is orientable if and only if $G^A$ is orientable.
\end{enumerate}
\end{proposition}
Note that it follows from the proposition that partial duals may be formed one edge at a time. Also note that the form of Item~\ref{p.pd2.3} of the proposition is very similar to that of the second part of Lemma~\ref{le:delcondual1}. We will return to this later.

\subsection{Quasi-trees}\label{sss.qt}
Quasi-trees are one of our fundamental objects of study. They are the analogue of trees for ribbon graphs, and our terminology reflects this.
A \emph{quasi-tree} $Q$ is a connected ribbon graph with exactly one boundary component.
If $G$ is a connected ribbon graph, a \emph{spanning quasi-tree} $Q$ of $G$ is a spanning ribbon subgraph with exactly one boundary component.
For disconnected graphs, we abuse notation by saying that $Q$ is a \emph{spanning quasi-tree} of $G$ if $k(Q)=k(G)$ and the connected components of $Q$ are spanning quasi-trees of the connected components of $G$.

We record the following basic facts about quasi-trees for reference later. For~\eqref{l.3a.3}, recall that, for ribbon graphs, $E(G)=E(G^*)$.
\begin{lemma}\label{l.3a}
 Let $G$ be a ribbon graph, and  $Q$ be a spanning quasi-tree of $G$. Then the following hold.
\begin{enumerate}
\item \label{l.3a.1} $0\leq \gamma(Q) \leq \gamma(G)$.
\item \label{l.3a.2} $\gamma(Q)=0$ if and only if $Q$ is a maximal spanning forest of $G$.
\item \label{l.3a.3}  $(V(G),A)$ is a spanning quasi-tree of $G$ of Euler genus  $\gamma$ if and only if $(V(G^*),A^c)$ is a spanning quasi-tree of $G^*$ of Euler genus $\gamma(G)-\gamma$.
\item \label{l.3a.4} If $Q=(V(G),A)$ then $\gamma(Q)=\gamma(G)$ if and only if $(V(G^*),A^c)$ is a maximal spanning forest of $G^*$.
\end{enumerate}
\end{lemma}
\begin{proof}
Items \eqref{l.3a.1} and \eqref{l.3a.2} follow easily from Euler's formula. Item~\ref{l.3a.4} is an immediate consequence of \eqref{l.3a.2} and \eqref{l.3a.3}. It remains to prove \eqref{l.3a.3}. For this first assume that $G$ is connected. Consider the intermediate step of the formation of $G^*$ from $G$, as described in Section~\ref{ss.dual}, in which the holes of $G$ have been capped off with elements of $V(G^*)$ giving a surface  $\Sigma := V(G)\cup V(G^*)\cup E(G)$. For each $A\subseteq E(G)$, observe that  $V(G)\cup A=(\Sigma \ba V(G^*))\ba A^c $ and $V(G^*) \cup A^c=(\Sigma \ba V(G))\ba A$ have the same boundary components. Thus  $Q:=(V(G),A)$ is a spanning quasi-tree of $G$ if and only if
$Q':=(V(G^*),A^c)$ is a spanning quasi-tree of $G^*$.
Suppose that $Q$ and $Q'$ are both spanning quasi-trees. Then each of $Q$ and $Q'$ has one boundary component and is connected. Moreover $v(Q')=v(G^*)=f(G)$.
Euler's formula gives  $\gamma(Q)=2k(Q)-v(Q)+e(Q)-f(Q)= 1-v(G)+|A|$ and $\gamma(Q')=2k(Q')-v(Q')+e(Q')-f(Q')= 1-f(G)+e(G)-|A|$. Thus  $\gamma(Q)+\gamma(Q')=2-v(G)+e(G)-f(G)=\gamma(G)$. Extending the result to disconnected graphs is straightforward because each of the parameters $v$, $e$, $f$ and $k$ is additive over connected components, and the geometric dual of a disconnected ribbon graph is the disjoint union of the geometric duals of its connected components.
\end{proof}

\section{Delta-matroids from ribbon graphs}\label{s4}

\subsection{Defining the delta-matroids}\label{s4.1}

Consider a connected ribbon graph $G=(V,E)$. We start by considering some standard ways that $G$ gives rise to a matroid.
The most fundamental  matroid associated with $G$ is its cycle matroid
$M(G)=(E, \mathcal{B})$, where $ \mathcal{B}$ consists of the edge sets of the spanning trees of $G$. The matroid $M(G)$ contains no information about the topological structure of $G$, only its graphical structure.  This is because trees always have genus zero and therefore cannot depend upon the embedding of $G$. Our aim here is to find the matroidal analogue of an embedded graph, and to do this we clearly need to adapt the definitions of  $M(G)$. By thinking of the the construction of $M(G)$ in terms of ribbon graphs it becomes obvious how this should be done: spanning trees are genus-zero spanning ribbon subgraphs with exactly one boundary component, so to retain topological information, we drop the genus zero condition, consider quasi-trees instead of trees, and obtain the set system
$(E, \mathcal{F})$, where  $\mathcal{F}$ consists of the edge sets of the spanning quasi-trees of $G$.

There is a natural variation of the construction of a cycle matroid obtained by choosing $n\in \mathbb{N}_0$, taking  $E$ as the ground set and $\mathcal{B}$ to be either the edge sets formed by deleting $n$ edges from each spanning tree, or the edge sets formed by adding  $n$ edges to each spanning tree.
In the former case, $\mathcal{B}$ consists of the edge sets of spanning forests of $G$ having exactly $n+1$ connected components and $(E,\mathcal B)$ is shown to be a matroid by noting that it is the $n$th truncation of $M(G)$, see~\cite{Oxley11}. In the latter case, $(E,\mathcal{B})$ is the dual of the $n$th truncation of $M(G)^*$.
Consider this construction in terms of quasi-trees of ribbon graphs:
the number of boundary components is not determined by the number of edges added or removed and
can be anywhere between $1$ and $n+1$, if $n$ edges are added or removed. In the quasi-tree setting it no longer makes sense to make the distinction between adding and removing edges, as we did in the case of matroids and spanning trees.
These ribbon graph extensions of matroids naturally lead us to the make the following definition.
\begin{definition}
Let $G=(V,E)$ be a ribbon graph with $k(G)$ connected components, and let $n\in \mathbb{N}_0$. Then we define
\begin{enumerate}
\item $\F_{\leq n}(G):=   \{ A\subseteq E  \mid  f(A) \leq k(G)+n \}   $, and
\item  $\F_{n}(G):=   \{ A\subseteq E  \mid  f(A) =k(G)+n \}   $.
\end{enumerate}
\end{definition}

For a connected ribbon graph, $\F_{n}(G)$ is the collection of all edge sets that determine a spanning ribbon subgraph of $G$ with exactly $n+1$ boundary components, and $\F _{\leq n}(G)$ is the collection of all edge sets that determine a spanning ribbon graph of $G$ with at most $n+1$ boundary components. Note that $\F_{\leq 0}(G)=\F_{0}(G)$. This set will be particularly important to us here, and later we will denote it by just $\F(G)$. Note that $\F_n(G)$ may be empty.

\begin{example}\label{s4.ex1}
For the ribbon graph $G$ of Figure~\ref{f.descb},
\begin{align*}
\F_0(G)=&\F_{\leq 0}(G)=\{\{1\},\{2\},\{1,2,3\},\{1,2,4\}\},\\
\F_1(G)=&\{ \emptyset,  \{1,2\},\{1,3\},\{1,4\},\{2,3\},\{2,4\},\{1,2,3,4\}\},\\
 \F_2(G)=&\{   \{3\},\{4\},\{1,3,4\},\{2,3,4\}\},\\
  \F_3(G)=&\{  \{3,4\}\}, \text{ and}\\
  \F_n(G)=&\emptyset, \text{ for $n>3$}.
  \end{align*}
Then $\F_{\leq n}(G)$ can be found easily from these.
\end{example}

\begin{definition} For a ribbon graph  $G=(V,E)$ and a non-negative integer $n$, let
 $D_{\leq n}(G)$ denote the set system $(E, \F_{\leq n}(G))$, and $D_{n}(G)$ denote the set system $(E, \F_{n}(G))$.
\end{definition}

 \begin{theorem}\label{t.1}
 Let $G=(V,E)$ be a ribbon graph, and $n\in \mathbb{N}_0$. Then
 \begin{enumerate}
  \item \label{t.1c} $D_{\leq n}(G)=(E, \F_{\leq n}(G))$ is a delta-matroid, and
 \item  \label{t.1b} $D_1(G)=(E, \F_1(G))$ is a delta-matroid, if $G$ is non-empty and orientable.
 \end{enumerate}
 \end{theorem}

The proof of Theorem~\ref{t.1} follows from the next lemma.  For the next two proofs we use $G_A$ to denote the spanning ribbon subgraph $(V,A)$ of $G$. Note that $G_A$ does \emph{not} denote the induced ribbon subgraph $G|_A$.
 \begin{lemma}\label{l1}
 Suppose  $A\in \F_n(G)$,  $B\in \F_{\leq n}(G)$,  $e\in A\btu B$, and $A\btu e\not\in \F_{\leq n}(G)$. Then there exists  $f\in A\btu B$ such that  $A\btu \{e,f\}\in \F_{n}(G)$.
\end{lemma}
\begin{proof}
The ribbon graph $G_{A}$ has $n+k(G)$ boundary components and $G_{B}$ has at most  $n+k(G)$ boundary components. By Proposition~\ref{obv},  $f(A\btu e)$ and $f(A)$ differ by at most one.
Thus $G_{A\btu e}$ has $n+k(G)+1$ boundary components (as $A\btu e\not\in \F_{\leq n}(G)$).
We think of $G_{A\btu e}$ as a ribbon subgraph of $G_{A\cup B}$. We can then consider how the edges in $(A\btu B)\ba e$ meet the boundary components of $G_{A\btu e}$.

If there is an edge $f\in (B\ba A)\ba e$ that intersects two distinct boundary components of $G_{A\btu e}$, then adding this edge to $G_{A\btu e}$ will give a ribbon subgraph with one fewer boundary component, and so $A\btu \{e,f\}\in \F_{n}$.
If there is an edge $f\in (A\ba B)\ba e$ that  meets  two distinct boundary components of $G_{A\btu e}$, then removing  this edge from $G_{A\btu e}$ results in a ribbon subgraph with one fewer boundary component, and so $A\btu \{e,f\}\in \F_{n}(G)$.

All that remains is the case in which each edge in $(A\btu B)\ba e$ intersects exactly one boundary component of $G_{A\btu e}$.
We shall show that this case cannot happen.

To see why, observe that $G_B$ can be obtained from $G_{A\btu e}$ by first deleting the edges in $(A\ba B)\ba e$ and then adding the edges in $(B\ba A)\ba e$, one by one. Colour the boundary components of $G_{A\btu e}$ so that each one receives a different colour. Whenever an edge is added or deleted, the only boundary components that change are those intersecting an edge that is deleted or those intersecting the two line segments forming the ends of an edge that is added. At each step the number of boundary components may stay the same, or increase or decrease by one. After a step where the number of boundary components increases by one, the two new boundary components are given the same colour as the one they replace. We claim that when the number of boundary components decreases by one, the two boundary components being replaced have the same colour. The single boundary component replacing them may then be given this common colour. Suppose that the claim is not true and consider the first time that an edge $f$ is added or deleted in such a way that the number of boundary components decreases and the two boundary components $C_1$ and $C_2$ that are changed by the edge addition or deletion have different colours. Let $G'$ denote the ribbon graph obtained just before $f$ is added or deleted.
 Both $C_1$ and $C_2$ contain a line segment that is removed from the boundary of $G'$ after the addition or deletion of $f$.
 Let $L_1$ and $L_2$ denote these line segments. Then $L_1$ and $L_2$
 are part of the boundary of each ribbon graph in the process up to the current step, including $G_{A\btu e}$. Although the boundary components to which these line segments belong may change, their colours do not.
 As $f\in (A\btu B)\ba e$, it intersects exactly one boundary component of $G_{A\btu e}$.
Therefore $L_1$ and $L_2$ have the same colour in $G_{A\btu e}$, and consequently in $G'$.
Thus the claim follows and moreover all the original colours used to colour the boundary components of $G_{A\btu e}$ are used to colour the boundary components of $G_B$.
Therefore $G_B$ has at least as many boundary components as $G_{A\btu e}$. This contradicts our hypotheses from the statement of the  lemma  that $B\in \F_{\leq n}(G)$ and  $A\btu e\not\in \F_{\leq n}(G)$. %
\end{proof}

\begin{proof}[Proof of Theorem~\ref{t.1}]
In each case it is enough to show that the given families of feasible sets satisfy the Symmetric Exchange Axiom.

For Item~\ref{t.1c}, let $A,B\in \F_{\leq n}(G)$ and $e\in A \btu B$. If $A\btu e\in \F_{\leq n}(G)$, then taking  $f=e$ gives $A\btu \{e,f\}\in \F_{\leq n}(G)$, as desired. In the exceptional case, $A\btu e\notin \F_{\leq n}(G)$, so  it follows from
 Proposition~\ref{obv} that  $A\in \F_{n}(G)$.
Then Lemma~\ref{l1} guarantees that there is an element $f\in A\btu B$   such that  $A\btu \{e,f\}\in \F_{\leq n}(G)$.

For Item~\ref{t.1b}, we first observe that it follows easily from Euler's formula that the parity of $f(A)-f(B)$ is the same as the parity of $e(A)-e(B)$. In particular, the sizes of all spanning quasi-trees of $G$ have the same parity, and the sizes of all members of ${\mathcal F}_1$ have the opposite parity. By Proposition~\ref{obv}, we have $|f(A\btu e)-f(A)| \leq 1$. Thus, if $A \in \F_0(G)$ and $e \in E$, then $A\btu e \in {\mathcal F}_1(G)$, so $D_1$ is a proper set system. Let $A$, $B$ be members of ${\mathcal F}_1(G)$ and $e\in A \btu B$. If $A \btu e \notin \mathcal{F}_{\leq 1}(G)$, then by Lemma~\ref{l1}, there exists $f \in A\btu e$ such that $A\btu\{e,f\}\in \mathcal{F}_1(G)$. It remains to consider what happens if $A\btu e \in \mathcal{F}_0(G)$. As $|A|$ and $|B|$ have the same parity, there exists $f  \in (A\btu B)-e$. Now, by our earlier observation, $(A\btu e)\btu f \in \mathcal{F}_1(G)$. Hence $D_1(G)$ is a delta-matroid.
\end{proof}

In general,  the set system $D_n(G)$ is not a delta-matroid. For example, if $G$ is the plane graph obtained by taking a triangle with edges 1, 2, 3 and adding an edge 4 in parallel with edge 3, then  $\F_2(G)=\{ \emptyset, \{3,4\},\{1,2,3,4\}  \}$  and it is readily seen that   $D_2(G)$  is not a delta-matroid.  Also, if $G$ is non-orientable $D_1(G)$ may not be a delta-matroid.  Consider, for example, the ribbon graph $G$ of Euler genus 2 obtained by adding an interlaced non-orientable loop to a plane 2-cycle.

\subsection{Ribbon-graphic delta-matroids}\label{s4.2}
One of the main purposes of this article is to illustrate that the delta-matroid $D_{0}(G)=D_{\leq 0}(G)$ plays a role in delta-matroid theory analogous to
the role graphic matroids play in matroid theory. In this subsection we set up some additional terminology for these delta-matroids and show that they have appeared in the literature in other guises.

\begin{definition}
Let $G=(V,E)$ be a ribbon graph. We use $\F(G)$ to denote the set $\F_0(G)=\F_{\leq 0}(G)$, so that
\[\F(G):=\{F\subseteq E(G) \mid F\text{ is the edge set of a spanning quasi-tree of }G\},\]
and $D(G)=(E,\F)$ to denote the delta-matroid $D_0(G)=D_{\leq0}(G)$. We say that  $D(G)$ is a \emph{ribbon-graphic} delta-matroid.
\end{definition}

\begin{example}
For the ribbon graph $G$ of Figure~\ref{f.descb}, \[D(G)=(\{1,2,3\}, \{\{1\},\{2\},\{1,2,3\},\{1,2,4\}\}).\]
\end{example}

To relate the delta-matroid $D(G)$ to the literature, particularly to Bouchet's foundational work on delta-matroids, we take what may appear to be a detour into transition systems.
Let $F=(V,E)$ be a 4-regular graph. Each vertex $v$ of $F$ is incident with exactly four half-edges. A \emph{transition} $\tau_v$ at a vertex $v$ is a partition of the half-edges at $v$ into two pairs, and a \emph{transition system}, $\tau:=\{\tau_v\mid v\in V\}$ of $F$ is a choice of transition at each of its vertices.

For the purposes of this section, we allow graphs to include \emph{free loops}, that is edges which are not incident with any vertex. We think of a free loop as a circular edge or as a cycle on zero vertices.
Given a transition system $\tau$ of $F$, we can obtain a set of free loops as follows.  If $(u,v)$ and $(w,v)$ are two non-loop edges whose half edges are paired at the vertex $v$, then we replace these two edges with a single edge $(u,w)$. In the case of a loop, we temporarily imagine an extra vertex of degree two on the loop, carry out the operation, and then suppress the temporary vertex. Doing this replacement for each pair of half edges paired together in the transition system $\tau$ results in a set of free loops, that we denote by $F(\tau)$ and call a \emph{graph state}.

Since $F$ is 4-regular, at each vertex there are three transitions. Choose exactly two transitions $\tau_v$ and $\tau_v'$ at each vertex, and consider the set $\mathcal{T}$ consisting of all transition systems of $F$ in which the transition at each vertex $v$ is one of the distinguished transitions, $\tau_v$ or $\tau_v'$. An element of $\mathcal{T}$ is called an \emph{allowable transversal}. Fix some allowable transversal $T\in \mathcal{T}$, and let
 \[ D(F,\mathcal{T},T)=  (T , \{ \tau \cap T \mid \tau \in \mathcal{T}\text{ and } |F(\tau)|=k(F)  \}).  \]
Kotzig's Theorem~\cite{kotzig} implies that $D(F,\mathcal{T},T)$ is a proper set system.
Bouchet showed in~\cite{ab1} that $D(F,\mathcal{T},T)$ is a delta-matroid. A delta-matroid that can be obtained in this way is called an \emph{Eulerian delta-matroid}. (Note that although Bouchet never uses the term ``Eulerian delta-matroid'' in~\cite{ab1}, it is implied that this is the intended definition by his later work, such as~\cite{abcoverings}.)

Bouchet showed  that $D(G)$ is a delta-matroid, albeit using a different language. Following~\cite{ab2}, let $G$ be a connected graph cellularly embedded in a surface $\Sigma$, and let $G^*$ be its geometric dual. Consider the natural immersion of $G\cup G^*$ in $\Sigma$.
For each $B\subseteq  E(G)$ let $B^*$ denote the corresponding set in $E(G^*)$. A set $B\subseteq  E(G)$ is said to be a \emph{base} if $\Sigma - \mathrm{cl}(B\cup (B^c)^*)$ is connected, where $\mathrm{cl}$ denotes the topological closure operator. Let $\mathcal{F}_b(G)$ denote the collection of all bases of $G$. Bouchet showed that $\mathcal{F}_b(G)$ satisfies the Symmetric Exchange Axiom, and so the pair $D_{\cell}(G)=(E,\mathcal{F}_b(G))$ is a delta-matroid.

By changing from the language of cellularly embedded graph to ribbon graphs we can see that $D(G)$ and  $D_{\cell}(G)$ are identical objects. To see this consider $G\subset \Sigma$ and $G^*\subset \Sigma$ as ribbon graphs $G'$ and $G'^*$ respectively. Then $\Sigma = V(G') \cup V(G'^*) \cup E(G)$ as described in Section~\ref{ss.dual}. It is not hard to see that  the number of components of $\Sigma - \mathrm{cl}(B\cup (B^c)^*)$ is exactly the number of boundary components of $G'\ba B^c$. It follows that $B$ defines a base of $G\subset \Sigma$ if and only $(V(G'),B) $ is a spanning quasi-tree of $G'$. Thus $D(G)$ and $D_{\cell}(G)$ coincide.

Bouchet did not use the language of quasi-trees to show that $D_{\cell}(G)$ is a delta-matroid, but rather transition systems and Eulerian delta-matroids, identifying it with a construction from~\cite{ab1}. For this, again let $G$ be a connected graph cellularly embedded in a surface. Its {\em medial graph}, $G_m$, is the embedded graph constructed by  placing a vertex on each edge of $G$, and then drawing the edges of the medial graph by following the face boundaries of $G$ (so  each vertex of $G_m$ is of degree $4$). The medial graph of an isolated vertex is a free loop. The vertices of $G_m$ are 4-valent and correspond to the edges of $G$. Every medial graph has a  \emph{canonical face 2-colouring} given by colouring faces corresponding to a vertex of $G$ black, and the remaining faces white. We can use the canonical face 2-colouring to distinguish among the three types of vertex transitions. We call a vertex transition \emph{white} if it pairs half-edges that share a white face, \emph{black} if it pairs half-edges that share a black face, and \emph{crossing} otherwise. If $\mathcal{T}_m$ consists of all the transition systems that have only white or black transitions at each vertex, and $W$ consists only of the white transitions, then it is not hard to see that $D(G)=D(G_m,\mathcal{T}_m, W)$.

This discussion shows that every ribbon-graphic delta-matroid is Eulerian. In fact, ribbon-graphic delta-matroids are exactly Eulerian delta-matroids.
\begin{theorem}[Bouchet~\cite{ab2}]
\label{ribboneuler}
A delta-matroid $D$ is Eulerian if and only if $D\cong D(G)$, for some ribbon graph $G$.
\end{theorem}
\begin{proof}[Sketch of proof.]
If $D$ is Eulerian then, by definition, we can obtain it as some $D(F,\mathcal{T},T)$. We need to find a ribbon graph $G$ such that  $D=D(G_m,\mathcal{T}_m, W)$. But such a ribbon graph can be obtained as a cycle family graph of $F$, from~\cite{EMM}. (The cycle family graphs of $F$ are precisely the embedded graphs that have a medial graph isomorphic to $F$.) The six choices at each vertex in the construction of a cycle family graph correspond to the six choices of the white and black transitions of $G_m$ (c.f. the proof of Theorem~4.12 of~\cite{EMM}).
\end{proof}

We have just seen that the delta-matroids of ribbon graphs considered here appeared in a rather different framework as Eulerian delta-matroids in Bouchet's initial work on delta-matroids. Here, we are proposing that for many purposes, the class of Eulerian delta-matroids, and delta-matroid theory in general, is best thought of as extensions of ribbon graph theory. (Saying this, of course there are certainly situations where it is most helpful to think of Eulerian delta-matroids as generalisations of transition systems.)  As we will demonstrate here, this is because there is a natural and fundamental compatibility between ribbon graph theory and delta-matroid theory, with many constructions, results, and proofs in the two areas being translations of one another.

From the perspective of Eulerian delta-matroids, $D(G_m,\mathcal{T}_m, W)$ is significant since the transition systems of $G_m$ arise  canonically. Another setting in which canonical transition systems arise is in digraphs. Suppose that $\vec{F}$ is a 4-regular digraph with two incoming and two outgoing half-edges at each vertex.  At each of its vertices there are two natural transitions that are consistent with the direction of the half-edges of the digraph. We take $\vec{\mathcal{T}}$ to be the set of all transition systems that arise from these choices. Then for each $\vec{T}\in \vec{\mathcal{T}}$, $D(\vec{F},\vec{\mathcal{T}}, \vec{T})$ is a delta-matroid. We call a delta-matroid arising in this way a \emph{directed Eulerian delta-matroid}.

\begin{theorem}[Bouchet~\cite{ab2}]
\label{dieul}
A delta-matroid $D$ is directed Eulerian if and only if $D= D(G)$, for some orientable ribbon graph $G$.
\end{theorem}
\begin{proof}[Sketch of proof.]
First suppose that $D= D(G)$, for some orientable ribbon graph $G$. Arbitrarily orient (the surface) $G$ and draw its canonically face 2-coloured medial graph $G_m$ on it. Direct  each edge of $G_m$ so that it is consistent with the orientation of the black face it bounds.

Conversely, suppose that $D$ is directed Eulerian, arising from a digraph $\vec{F}$. By the proof Theorem~\ref{ribboneuler}, we know $D=D(G_m,\mathcal{T}_m, W)$ for some ribbon graph $G$, where the underlying graphs of $G_m$ and $\vec{F}$ are isomorphic. The direction of $\vec{F}$ induces a direction of $G_m$. Furthermore, by forming the twisted duals (see~\cite{EMM}) $G^{\tau(e)}$ or $G^{\tau\delta(e)}$, if necessary, we may assume that the transitions that are consistent with the directions of $\vec{F}$ coincide with the black and white transitions of $G_m$. These directions induce an orientation on each black face of $G_m$, and hence of each vertex and half-edge of $G$. Since the black and white transitions of $G_m$ are consistent with   transitions coming from the directions of $\vec{F}$, these orientations of vertices must be consistent and so $G$ is orientable.
\end{proof}

Combining Theorems~\ref{ribboneuler} and~\ref{dieul}, and using the fact from Proposition~\ref{p.4b}  that $D(G)$ is even if and only if $G$ is orientable,  immediately gives the following.
\begin{corollary}[Bouchet~\cite{ab2}]\label{cor:beforetraldi}
A delta-matroid $D$ is directed Eulerian if and only if it is Eulerian and even.
\end{corollary}

In recent papers, Traldi  introduced the transition matroid of an abstract four-regular graph~\cite{Tr15trans} and the isotropic matroid of a symmetric binary matrix~\cite{Tr15iso}.
These two matroids have almost identical definitions: both are binary matroids described by a representation, with the only difference being a permutation of some of the columns labels.
Moreover, both are relevant to ribbon graphs. We have described the fundamental relationship between a ribbon graph and its medial graph, which is an embedded four-regular graph; in Section~\ref{s5.7} we describe how  a ribbon graph with one vertex may be represented by a symmetric binary matrix. In~\cite{Tr:CII} Brijder and Traldi  describe the construction of the transition matroid of a ribbon graph. We now describe the almost identical construction of the isotropic matroid of a ribbon graph, and discuss the extent to which it determines the ribbon graph.

Let $G=(V,E)$ be a connected ribbon graph and $G_m$ be its canonically face 2-coloured medial graph. 
Let $T$ be a transition system in $\mathcal{T}_m$ with $|G_m(T)|=1$.
In other words, $T$ defines an Eulerian circuit $C(T)$ in $G_m$ with no crossing transitions. Apply an orientation to the edges of $G_m$, so that $C(T)$ is now a directed Eulerian cycle.

We say that  two vertices $u$ and $v$ of $G_m$ are \emph{interlaced} with respect to $T$ if they are met in the cyclic order  $u\,v\,u\,v$ when travelling round $C(T)$. Let $A(G,T)$ denote the binary $|E|$ by $|E|$ matrix whose rows and columns are indexed by the elements of $E$. The $(e,e)$-entry of $A(G,T)$ is zero if and only if in $G_m$, opposite edges at the vertex corresponding to $e$ have inconsistent orientations in $C(T)$. For $e\neq f$, the $(e,f)$--entry is one if and only the vertices corresponding to $e$ and $f$ in $G_m$ are interlaced with respect to $T$.

We now let $IAS(G,T)$ be the $|E| \times 3|E|$ matrix
\[ \big(I \mid A(G,T) \mid I+A(G,T)\big).\]
The \emph{isotropic matroid} of $G$ is the binary matroid $M[IAS(G,T)]$ with representation $IAS(G,T)$.
Each edge of $G$ indexes three columns of $IAS(G,T)$, one in each of the three blocks, with the order of the indices consistent with the indices of $A(G,T)$.
Following Traldi, we use $e_\phi$, $e_\chi$ and $e_\psi$ to denote the columns of $IAS(G,T)$ corresponding to $e$ in $I$, $A(G,T)$ and $I+A(G,T)$ respectively. For $\nu \in \{\phi,\chi,\psi\}$, let $E_\nu=\{e_\nu\mid e\in E\}$.
A basis of $M[IAS(G,T)]$ is called \emph{transverse} if for each $e\in E$, it contains precisely one of $e_\phi$, $e_\chi$ and $e_\psi$.

The isotropic matroid itself does not determine $D(G)$, because knowledge of $T$ is required. Let $T_w$ denote the edges of $G$ where, at the corresponding vertex of $G_m$, $T$ takes the white transition.
Then from the discussion above $T_w$ is a feasible set of $D(G)$. We claim that $T_w \btu F$ is a feasible set of $D(G)$ if and only if the principal submatrix of $A(G,T)$ corresponding to the edges of $F$ is non-singular. This is easily verified when $|F| \leq 2$, by considering the effect of switching the transitions of $T$ from black to white or vice versa at the vertices of $G_m$ corresponding to edges in $F$. Results of Bouchet presented as Lemmas~\ref{for1} and~\ref{for2}, and Theorem~\ref{ribbonisbinary} in Section~\ref{s5.7} show that this is enough to verify the claim.
Thus there is a bijection between transverse bases of $M[IAS(G,T)]$ which do not intersect $E_\psi$ and feasible sets of $D(G)$ associating a basis $B$ with the feasible set $(B \cap E_\chi) \btu T_w$.

In~\cite{Tr15iso}, Traldi introduces the isotropic matroid of a symmetric binary matrix $A$, which has a representation of the same form as above, that is
\[ (I \mid A \mid I+A).\]
In particular in~\cite[Theorem 7]{Tr15iso} he describes exactly when two binary symmetric matrices have  isomorphic isotropic matroids.
To translate this result to ribbon graphs, requires the notion of twisted duality from~\cite{EMM}. Two ribbon graphs are twisted duals of each other if and only if their medial graphs are isomorphic as abstract graphs. Given a connected ribbon graph $G$ and a spanning quasi-tree $Q$ of $G$, let $T(Q)$ denote the transition system of $G_m$ taking the white transition at vertices of $G_m$ corresponding to edges of $Q$ and the black transition otherwise.
\begin{theorem}
Let $G_1$ and $G_2$ be connected ribbon graphs and let $Q_1$ and $Q_2$ be spanning quasi-trees of $G_1$ and $G_2$ respectively.
Then $IAS(G_1,T(Q_1)) \simeq IAS(G_2,T(Q_2))$ if and only if $D(G_1) \simeq D(G_3)$ for some twisted dual $G_3$ of $G_2$.
\end{theorem}

In Section~\ref{s5.7} we discuss binary delta-matroids, which arise from binary symmetric matrices. Further results from~\cite{Tr15iso} describe how any binary delta-matroid can be viewed as an isotropic matroid.

\subsection{The spread of a delta-matroid}\label{s4.3}

In Section~\ref{s4.1} we associated a family of delta-matroids to a ribbon graph.
In this section we introduce an operation on delta-matroids that enables us to relate $D_{\leq n}(G)$ to $D(G)$.

\begin{definition}
\label{kspreaddef}
Let $D=(E,\mathcal{F})$ be a delta-matroid and $n$ be a non-negative integer. Then we define $\F _{\leq n}$ by
\[\F _{\leq n}:=\{F\btu A \mid F\in\F\text{ and }A\subseteq E\text{ and }|A|\leq n\}.\]
We say that the set system   $D_{\leq n}:=(E,\mathcal{F}_{\leq n})$ is the \emph{$n$-spread of $D$}.
\end{definition}

Note that $D_{\leq 0}=D$. In order to show that  $D_{\leq n}$ is a delta-matroid, we will define delta-matroid sum. This sum is not the same concept as the direct sum, which we define later. We will only refer to the sum in this section, so confusion should not arise.
If $D=(E,\F)$ and $D'=(E,\F')$ are proper set-systems then their \emph{sum} is the set system $(E,\F \btus \F')$ where
\[ \F\btus \F' := \{ F\btu F' \mid F \in \F \text{ and } F' \in \F'\}.\]
Bouchet and Schw\"arzler~\cite{BS98} attribute the following result to Duchamp. A proof of the corresponding result for jump systems may be found in~\cite{BC95} and it is easy to translate this proof to delta-matroids.
\begin{theorem}\label{th:sum}
If $D$ and $D'$ are delta-matroids, then their sum is also a delta-matroid.
\end{theorem}

\begin{proposition}
\label{k-spread-delta}
If $D=(E,\F)$ is a delta-matroid and $n$ a non-negative integer, then $D_{\leq n}$ is a delta-matroid.
\end{proposition}
\begin{proof}
Let $O_n=(E,\{\emptyset\}_{\leq n})$. Then it is clear that $O_n$ is a delta-matroid and that $D_{\leq n}$ is the sum of $D$ and $O_n$. The result follows from Theorem~\ref{th:sum}.
\end{proof}

\begin{remark}\label{rem:wiggle}
  Theorem~\ref{th:sum} can be used to generate interesting families of delta-matroids.  The \emph{uniform matroid}, denoted by $U_{r,m}$, is a matroid with $m$ elements in the ground set and rank $r$, such that every subset of the ground set with $r$ elements is a basis.  An interesting family of delta-matroids may be constructed by taking the sum of a delta-matroid $D$ with the uniform matroid of rank $r$ defined on the ground set of $D$.  This gives a delta-matroid in which a set $F$ is feasible if and only if there is a feasible set $F'$ of $D$ with $|F\btu F'|=r$.
\end{remark}

The following is an easy observation concerning spreads.
\begin{proposition}\label{prop:twistspreadcommute}
If $D=(E,\F)$ is a delta-matroid, $n$ is a non-negative integer and $A$ is a subset of $E$ then
$(D *A)_{\leq n} = D_{\leq n} *A$.
\end{proposition}
\begin{proof}
A set is feasible in the $n$-spread of $D$ if and only if it is feasible in $D\ast X$ for some $X$ with  $|X|\leq n$.
That is,
\[\mathcal{F}_{\leq n} = \bigcup_{\substack{ X \subseteq E \\ |X|\leq n}} \mathcal{F}(D*X).\]
Thus
\[\F(D_{\leq n}* A)= \bigcup _{\substack{ X \subseteq E\\|X|\leq n}}\F ((D*X)*A)=\bigcup _{\substack{X \subseteq E\\ |X|\leq n}}\F ((D*A)*X)=\F((D*A)_{\leq n}).\]
\end{proof}

\begin{definition}
\label{ktoggledef}
Let $D=(E,\F)$ be a delta-matroid and $n$ a non-negative number. Then we define $\F _{\btu n}$ as $\F _{\leq n}- \F _{\leq n-1}$.  The \emph{$n$-toggle of $D$}, which is denoted by $D_{\btu n}$, is defined to be $(E,\F_{\btu n})$.
\end{definition}

Note that $D_{\btu 0}=D$.

\begin{proposition}
\label{eventoggle}
Let $D=(E,\F)$ be an even delta-matroid with $E \ne \emptyset$.
Then $D_{\btu 1}$ is a delta-matroid.
\end{proposition}
\begin{proof}
Take $A$ and $B$ in $\F _{\btu 1}$ and $x$ in $A\btu B$. Then $A$ and $B$ are in
$\F _{\leq 1}$.
By Proposition~\ref{k-spread-delta}, there is an element $y$ in $A\btu B$ such that $A\btu \{x,y\}\in\F _{\leq 1}$. If $y \ne x$ then $A \btu \{x,y\} \in \F _{\btu 1}$, so we may assume that $y=x$. In this case  we must have $A \btu x \in \F$. Now $|A \btu B| \geq 2$, because $D$ is even, so we may choose $z\in (A \btu B)-x$. Clearly $A \btu \{x,z\} = (A\btu x) \btu z$ is in $\F_{\btu 1}$.
\end{proof}

The following theorem shows that $D_{\leq n}(G)$ and $D_{ n}(G)$ can be obtained from $D(G)$ by $n$-spreads and $n$-toggles.
\begin{theorem}
\label{spreadisspread}
Let $G=(V,E)$ be a ribbon graph and $n$ a non-negative number.
Then
\begin{enumerate}
\item \label{spreadisspread.1} $D_{\leq n}(G)$ is the $n$-spread of $D(G)$, that is  $D_{\leq n}(G)= D(G)_{\leq n}$; and
\item \label{spreadisspread.2} $D_n(G)=D(G)_{\btu n}$.
\end{enumerate}
\end{theorem}
\begin{proof}
Item \eqref{spreadisspread.2} follows directly from \eqref{spreadisspread.1}, since $\F_n(G)=\F_{\leq n}(G) - \F_{\leq n-1}(G)$, and
$\F(D(G)_{\btu n} )=    \F(D(G)_{\leq n}) - \F(D(G)_{\leq n-1})$.
Thus it suffices to show that \eqref{spreadisspread.1} holds.

We will show that
\begin{sublemma}
\label{forward}
$\F _{\leq n}(G)$ is contained in the feasible sets of the $n$-spread of $D(G)$.
\end{sublemma}
We proceed using induction on $n$. Clearly the result is true when $n=0$.
Take $F\in\F _{\leq n}(G)$.
Suppose there is an edge $e\in F$ such that $e$ is incident with two boundary components of $(V,F)$.
Then $e$ is not a bridge, so $f(F\btu e)-k(G)=(f(F)-1) -k(G)$.
Hence $F\btu e$ is in $\F _{\leq n-1}$.
By induction, we know that $F\btu e$ is a feasible set in the $(n-1)$-spread of $D(G)$.
Hence $F\btu e = F'\btu A$, where $F'\in\F(G)$ and $|A|\leq n-1$.
Then $F=(F'\btu A)\btu e=F'\btu (A\btu e)$, so $F$ is in the $n$-spread of $D(G)$.
So we may assume that each connected component of $(V,F)$ has exactly one boundary component.
If $k(F) \ne k(G)$
then there is an edge $e$ of $G$ which is not in $F$, joining two connected components of $(V,F)$. Thus $f(F\btu e)-k(G)=(f(F)-1) -k(G)$ and the result follows in a similar way.
If $k(F)=k(G)$ then $F$ is a spanning quasi-tree in $G$.
Thus $F$ is in $\F(G)$, which is itself contained in the $n$-spread of $D(G)$ and~\ref{forward} holds.

We conclude this proof by showing that
\begin{sublemma}
\label{reverse}
the feasible sets in the $n$-spread of $D(G)$ are contained in $\F _{\leq n}(G)$.
\end{sublemma}
Again we proceed using induction on $n$. Clearly the result is true when $n=0$.
Take $F$ in the $n$-spread of $D(G)$. Then there is a spanning quasi-tree $F'$ of $G$ and a set $A$ with $|A| \leq n$ such that $F=F'\btu A$. If $A$ is empty, then there is nothing to prove, so let $a\in A$. Now $F\btu a=F'\btu (A-a)$ is in the $(n-1)$-spread of $D(G)$ and, by induction, is contained in $\F_{\leq n-1}(G)$. Thus $f(F\btu a)-k(G) \leq n-1$. But, by Proposition~\ref{obv}, the number of boundary components of $F$ and $F\btu a$ differ by at most one. Hence $f(F)-k(G) \leq f(F\btu a)-k(G)+1 \leq n$, so $F \in \F _{\leq n}(G)$.  Thus~\ref{reverse} holds.
\end{proof}

Two natural questions arise from the preceding results. Is $D_n$ a delta-matroid for $n\geq 2$?
Can the evenness condition be dropped from Proposition~\ref{eventoggle}? Both questions have negative answers.
 We saw at the end of Section~\ref{s4.1} an example that showed that in general $D_2(G)$, which equals $D(G)_{\btu 2}$, is not a delta-matroid. Also the example given there showing that $D_1(G)$, which equals $D(G)_{\btu 1}$, may not be a delta-matroid  shows that evenness cannot be dropped. (We will shortly see (Proposition~\ref{p.4b}) that $D(G)$ is even if and only if $G$ is orientable.)
The class of delta-matroids whose $1$-toggle is a delta-matroid may be a nice class. It would be interesting to have a characterisation of it.

\section{Delta-matroids and ribbon graphs:  geometric interplay}\label{s5}

\subsection{Duals, partial duals and twists}\label{s5.1}
Recall from Section~\ref{matroidsanddelta-matroids} that, if $D=(E,{\mathcal{F}})$ is a delta-matroid and $A\subseteq E$, then the \emph{twist} of $D$ with respect to $A$, is the delta-matroid $D* A:=(E,\{A\bigtriangleup X\mid X\in \mathcal{F}\})$.
In particular, the dual $D^*$ of $D$ is equal to $D*E$. Thus we may regard a twist $D\ast A$ as being a `partial dual' of a delta-matroid in the sense that the dual is `formed with respect to only the elements in $A$'. The following theorem shows that this notion of partial duality corresponds exactly to partial duality of ribbon graphs (see Section~\ref{ss.dual}). That is, on the delta-matroid level, twisting and partial duality are equivalent.  Although this is a fairly simple result, it will prove to be extremely useful and important in what follows.

\begin{theorem}\label{t.twdu}
Let $G=(V,E)$ be a ribbon graph, $A\subseteq E$ and $e\in E$. Then
$D_{\leq k}(G^A)=D_{\leq k}(G)*A$ and, in particular, $D(G^A)=D(G)*A$. Furthermore, if $G$ is orientable, then $D_{1}(G^A)=D_{1}(G)*A$.
\end{theorem}

\begin{proof}
We will first prove the statement for $D(G)$. It is enough to prove it for $A=\{e\}$. We need to show  for each $Q\subseteq E$ that $ (V(G), Q )$ is a spanning quasi-tree of $G$ if and only if  $ (V(G^e), Q\triangle e )$ is a spanning quasi-tree of $G^{e}$. But this follows immediately upon observing that in Table~\ref{tablecontractrg}, in all cases, $G$ and $G^{e}\setminus e$, as well as $G\setminus e$ and $G^{e}$ have the same number of  boundary components.

The general statement follows directly from the facts that $D_{\leq k}(G)$ is the $k$-spread of $D(G)$ and the $k$-spread and twisting commute.
These facts are established by Theorem~\ref{spreadisspread} and by Proposition~\ref{prop:twistspreadcommute}, respectively.
\end{proof}

For matroids $M(G^*)=M(G)^*$  when $G$ is a plane graph. However, this identity does not hold for non-plane graphs. The following corollary, which is obtained by taking $A=E(G)$ in Theorem~\ref{t.twdu}, explains why this is. It shows that geometric duality is a delta-matroidal property, rather than a matroidal property.
The duality identity $M(G^*)=M(G)^*$  holds only for plane graphs because it is only in this case that  $M(G)$ and $D(G)$ coincide.
\begin{corollary}\label{c.gedu}
Let $G$ be a ribbon graph. Then
$D_{\leq k}(G^*)=D_{\leq k}(G)^*$ and, in particular, $D(G^*)=D(G)^*$. Furthermore, if $G$ is orientable, then $D_{1}(G^*)=D_{1}(G)^*$.
\end{corollary}

\subsection{Seeing ribbon graph structures in a delta-matroid}\label{s5.2}

Next we show that basic topological information about $G$ can be recovered from its delta-matroid.
Because of the connection with Bouchet's work that we have established, we could derive Item~\ref{p.4b.4} in the following proposition from~\cite[Theorem~5.3]{ab2}, but we instead give a direct proof for completeness.
We also give a short proof for Item~\ref{p.4b.3}, although it follows from~\cite[Theorem~4.1(iv)]{ab2}.

\begin{proposition}\label{p.4b}
Let $G$ be a ribbon graph and let $D=D(G)$.
\begin{enumerate}
\item  \label{p.4b.1} The feasible sets of $D$ with cardinality $m$ are in 1-1 correspondence with the spanning quasi-trees of $G$ with Euler genus  $m-v(G)+k(G)$.
\item \label{p.4b.2} The rank of $D_{\min}$ is equal to the size of a maximal spanning forest, that is, $r(D_{\min})= v(G)-k(G)$.
\item\label{p.4b.3}  The width of a ribbon-graphic delta-matroid is equal to the Euler genus of the underlying ribbon graph, that is, $  \gamma(G)=\spn(D)$.
\item\label{p.4b.4} The delta-matroid $D$ is even if and only if $G$ is orientable.
\end{enumerate}
\end{proposition}
\begin{proof}
The one-to-one correspondence in~\eqref{p.4b.1} follows immediately from the definition of $D$.
Take $F\in\F(D)$ and let $Q$ be the corresponding spanning quasi-tree.
Let $m=|F|$.  Then $e(Q)=m$.  Furthermore, $v(Q)=v(G)$ and $f(Q)=k(G)$.
Euler's formula gives $ \gamma(Q) = m-v(G)+k(G)$.
This completes the proof of~\eqref{p.4b.1}.

Now~\eqref{p.4b.1} implies that $m$ is minimized (respectively maximized) whenever $\gamma(Q)$ is minimized (respectively maximized).
Thus, if $m=r(D_{\min})$, then by applying Lemma~\ref{l.3a} we obtain $\gamma(Q)=0$ and that $Q$ is a maximal spanning forest of $G$. Moreover $v(G)-k(G)=r(D_{\min})$. Thus \eqref{p.4b.2} holds.

On the other hand, if $m=r(D_{\max})$ then $\gamma(Q)$ is maximized, so by applying Lemma~\ref{l.3a} again we deduce that $\gamma(Q)=\gamma(G)$.
Thus~\eqref{p.4b.3} holds.

Finally, we show that~\eqref{p.4b.4} holds.  Suppose that $G$ is orientable. Then every ribbon subgraph is orientable and so $\gamma(Q)$ is even for each spanning quasi-tree $Q$ of $G$. It follows from   \eqref{p.4b.1} that $|F|-r(D_{\min})$ is even for each feasible set $F$, and so the size of each feasible set has the same parity and $D$ is even.

If $G$ is non-orientable then it contains an non-orientable cycle $C$. Let $e$ be an edge of $C$. Then $C-e$ may be extended to a maximal spanning forest $F$ not containing $e$. But $F \cup e$ is also a spanning quasi-tree of $G$. Thus $D$ has feasible sets with cardinalities of both parities, so it is odd.
\end{proof}

Recall that, if $D$ is a delta-matroid, then $D_{\min}$ and $D_{\max}$ are matroids.
The properties from Proposition~\ref{p.4b} allow us to recognise $D(G)_{\min}$ and $D(G)_{\max}$ in terms of cycle matroids associated with $G$.
The following corollary can be recovered from~\cite{ab2}, but we give an independent proof here for completeness.

\begin{corollary}\label{c.4c}
Let $G$ be a ribbon graph. Then
\begin{enumerate}
\item \label{c.4c.1}  $D(G)_{\min}=M(G)$;
\item  \label{c.4c.4} $D(G)_{\max}=(M(G^*))^*$;
\item\label{c.4c.2} $D(G)=M(G)$ if and only if $G$ is a plane ribbon graph, otherwise $D(G)$ is not a matroid.
\end{enumerate}
\end{corollary}
\begin{proof}
By Proposition~\ref{p.4b},  the feasible sets of $D(G)_{\min}$  are exactly the edge sets of the genus-zero spanning quasi-trees of $G$. By Lemma~\ref{l.3a}, these are the edge sets of the maximal spanning forests of $G$. Thus they are exactly the bases of $M(G)$.  Thus~\eqref{c.4c.1} holds.

Next we prove \eqref{c.4c.4}.  Proposition~\ref{p.4b} implies that $F$ is a feasible set in $D(G)_{\max}$ if and only if $(V(G),F)$ is a spanning quasi-tree of $G$ of genus $\gamma(G)$, which by Lemma~\ref{l.3a}\eqref{l.3a.3} occurs if and only if $(V(G^*),F^c)$ is a spanning tree of $G^*$. Then \eqref{c.4c.1} implies that this holds exactly when $F^c$ is a feasible set in  $D(G^*)_{\min}= M(G^*)$.
The result follows.

Finally, we consider~\eqref{c.4c.2}.  The ribbon graph $G$ is plane if and only if $\gamma (G)=0$.  By Proposition~\ref{p.4b}\eqref{p.4b.3}, this occurs exactly when $\spn (D(G))=0$. But if $w(D(G))=0$, then $D(G)=D(G)_{\min}=M(G)$. If $w(D(G))>0$, then $D(G)$ has feasible sets of different sizes and cannot be a matroid.
\end{proof}

A consequence of Corollary~\ref{c.4c} is that, for a ribbon graph $G$, the spanning quasi-trees of minimal  genus, and of maximal genus, both give rise to matroids. It is natural to ask if the edge sets of spanning quasi-trees of any fixed genus form the bases of a matroid. Although these sets are equicardinal, it is not hard to see that this is not the case in general. For example, while $(\{1,2,3,4\},\{\emptyset,\{1,2\},\{3,4\},\{1,2,3,4\}\})$ is a delta-matroid, the set system $(\{1,2,3,4\},\{\{1,2\},\{3,4\}\})$ is not a matroid.

We now consider when some other classes of delta-matroids that we have defined in terms of ribbon graphs are matroids.
\begin{proposition}
\label{prop4.6}
Let $G=(V,E)$ be a connected ribbon graph.
\begin{enumerate}
\item \label{prop4.6.1} If $G$ is orientable, then $D_1(G)$ is a matroid if and only if one of the following occurs:
\begin{enumerate}
\item $G$ is a tree, hence $D_1(G)\cong U_{|E|-1,|E|}$; or
\item $G$ is a collection of trivial orientable loops on one vertex, hence $D_1(G)\cong U_{1,|E|}$; or
\item $G$ is a pair of interlaced orientable loops on one vertex, so $D_1(G)\cong U_{1,2}$.
\end{enumerate}
\item \label{prop4.6.2} If $D_{\leq k} (G)$ is a matroid for some integer $k \ge 1$, then $G$ comprises a single vertex and no edges.
\end{enumerate}
\end{proposition}
\begin{proof}
For \eqref{prop4.6.1}, suppose there exists $F$ in $\F(G)$ such that $F \neq \emptyset$ and $F \neq E$. Then, as $G$ is orientable, Theorem~\ref{spreadisspread}\eqref{spreadisspread.2}  and Proposition~\ref{p.4b}\eqref{p.4b.4}  imply that $D_1(G)$ has feasible sets of size $|F|-1$ and of size $|F|+1$.  Hence $D_1(G)$ is not a matroid in the case that $\F(G)$ is not contained in $\{\emptyset ,E\}$.
If $\F(G)= \{ \emptyset \}$, then $G$ is a collection of trivial orientable loops all connected to the same vertex, embedded in the sphere. In this case $D_1(G)$ is the uniform matroid of rank one, namely $U_{1,|E|}$. If $\F(G)=\{ E\}$, then $G$ is a tree and $D_1(G) = U_{|E|-1,|E|}$ is the uniform matroid of rank $|E|-1$.  Finally, suppose that $\F(G)=\{\emptyset,E\}$.  For $e\in E$, Axiom~\ref{sea} implies that there is an element $f\in E$ such that $\emptyset \btu\{e,f\}$ is feasible.  Hence $E=\{e,f\}$. As, $G$ is orientable, we must have $f\ne e$, so $G$ consists of a pair of interlaced orientable loops and $D_1(G)=\{\{e,f\},\{\{e\},\{f\}\}\}$ is a matroid isomorphic to $U_{1,2}$.
The reverse implication is easily checked. Hence \eqref{prop4.6.1} holds.

Now we show that \eqref{prop4.6.2} holds. Suppose that $(V,F)$ is a spanning tree of $G$ and $e\in E$. Then $(V,F\btu e)$ has at most two boundary components, so both $F$ and $F\btu e$ are in $\F_{\leq k}(G)$. Therefore if $E\ne \emptyset$ then $D_{\leq k}(G)$ is not a matroid. As we assumed that $G$ is connected, it must comprise a single vertex and no edges.
\end{proof}

\subsection{Loops, coloops, and ribbon loops}\label{s5.3}

For a graph $G$, it is well-known that an element $e$ is a loop or coloop in $M(G)$ if and only if $e$ is loop or bridge, respectively, in $G$. One would expect such a relation to hold for ribbon graphs and their delta-matroids, and the following proposition shows that indeed it does.  However,  while coloops in $D(G)$ correspond directly to bridges in a ribbon graph $G$, one has to be a little more careful in the case of loops. The difficulty is that, unlike graphs, ribbon graphs have different types of loops, orientable or non-orientable, and trivial or non-trivial.  Loops in $D(G)$ do not correspond to loops in $G$ in general, but rather to  trivial orientable loops in $G$. 

\begin{lemma}\label{l.4coloop}
Let $G$ be a ribbon graph, $D(G)=(E,\mathcal{F})$, and $e\in E(G)$. Then
\begin{enumerate}
\item \label{l.4coloop.1} $e$ is a coloop  in $D(G)$ if and only if  $e$ is a bridge in $G$; and
\item $e$ is a  loop in $D(G)$ if and only if  $e$ is a trivial orientable loop in $G$.
\end{enumerate}
\end{lemma}
\begin{proof}
For the first item, if $e$ is a bridge of $G$, then any ribbon subgraph of $G$ not containing $e$ has more
connected components than $G$ and therefore has more than $k(G)$ boundary components and is not a spanning quasi-tree. Thus if $e$ is a bridge it appears in every feasible set of $D(G)$ and so is a coloop. Conversely, if $e$ is a coloop  in $D(G)$ then it appears in every spanning quasi-tree of $G$. In particular, it appears in every spanning tree of $G$ and is therefore a bridge.

For the second item, $e$ is a trivial orientable loop in $G$ if and only if $e$ is a bridge in $G^*$.  Corollary~\ref{c.gedu} and item~\ref{l.4coloop.1} imply that this occurs if and only if $e$ is a coloop  in $D(G^*)=D(G)^*$.
This holds if and only if $e$ is a  loop in $D(G)$.
\end{proof}

We have seen that loops in ribbon graphs can be classified into several types. It turns out that that this classification may be usefully extended to elements of delta-matroids in general.
\begin{definition}\label{def:loops}
Let $D=(E,\mathcal{F})$ be a delta-matroid.
\begin{enumerate}
\item An element $e$ of $E$ is a {\em ribbon loop} if $e$ is a loop in $D_{\min}$.
\item A ribbon loop $e$ is  \emph{non-orientable} if $e$ is a ribbon loop in $D\ast e$ and is \emph{orientable} otherwise.
\item An orientable ribbon loop $e$ is \emph{trivial} if $e$ is in no feasible set of $D$ and is \emph{non-trivial} otherwise.
\item A non-orientable ribbon loop $e$ is \emph{trivial} if $F\btu e$ is in $\F$ for every feasible set $F\in\F$ and is \emph{non-trivial} otherwise.
\end{enumerate}
\end{definition}

If $e$ is a loop in $D$ then it is a ribbon loop of $D$, but the converse is not true in general. In fact, $e$ is a loop in $D$ if and only if it is a trivial orientable ribbon loop of $D$.

We now show that the various types of loops in a ribbon graph $G$ correspond to the various types of ribbon loops in the delta-matroid $D(G)$.

\begin{proposition}\label{p.loop2}
Let $G$ be a ribbon graph, $D=D(G)=(E,\mathcal{F})$, and $e\in E(G)$. Then
\begin{enumerate}
\item \label{p.loop2.a} $e$ is a  loop in $G$ if and only if  $e$ is a  ribbon loop in $D(G)$;
\item \label{p.loop2.b}  $e$ is an orientable loop in $G$ if and only if  $e$ is an orientable ribbon loop in $D(G)$;
\item \label{p.loop2.c}  $e$ is a trivial loop in $G$ if and only if  $e$ is a trivial ribbon loop in $D(G)$.
\end{enumerate}
\end{proposition}
\begin{proof}
We prove \eqref{p.loop2.a} first.
An edge $e$ is a loop of $G$ if and only if $e$ is an edge of no spanning tree of $G$.
This holds if and only if $e$ appears in no feasible set of $D_{\min}$.

Next we consider \eqref{p.loop2.b}.
From Table~\ref{tablecontractrg} we see that
a loop $e$ of $G$ is orientable if and only if it is not a loop of $G^e$.
By \eqref{p.loop2.a}, $e$ is not a loop of $G^e$ if and only if it is not a ribbon loop of $D(G^e)$. The result follows since  $D(G^e)=D*e$, by Theorem~\ref{t.twdu}.
Thus \eqref{p.loop2.b} holds.

For \eqref{p.loop2.c}, by Lemma~\ref{l.4coloop}, $e$ is a trivial orientable loop of $G$ if and only if $e$ is a loop of $D$ if and only if $e$ is a trivial orientable ribbon loop of $D$. It remains to deal with trivial non-orientable loops.

Suppose first that $e$ is a trivial non-orientable loop of $G$.
Take $F\in\F$. A trivial ribbon loop is not interlaced with any cycle of $G$ so $F \btu e \in \F$.  Hence $e$ is a trivial non-orientable ribbon loop in $D$.

Suppose finally that $e$ is a trivial non-orientable ribbon loop in $D$. It is enough to show that $e$ is trivial in $G$. Suppose that this is not the case. Take $C$ to be a cycle interlaced with $e$ and take $f\in C$. We may extend $C-f$ to a maximal spanning forest $F'$ of $G$.  As $F'$ contains no cycle, we know that $e\notin F'$ and $f\notin F'$. Now exactly one of $F' \cup f$ and $(F' \cup f)\btu e$ is a spanning quasi-tree of $G$, depending on whether or not $C$ is orientable, a contradiction. Thus $e$ is a non-trivial non-orientable loop of $G$ and \eqref{p.loop2.c} holds.
\end{proof}

\begin{lemma}
\label{3.5part2}
Let $D$ be a delta-matroid and $e$ an element of $D$.  Then $e$ is neither a coloop nor a ribbon loop in $D$ if and only if $e$ is a non-trivial orientable ribbon loop in $D*e$.
\end{lemma}
\begin{proof}
Suppose that $e$ is neither a coloop nor a ribbon loop of $D$. Then $e$ belongs to some basis of $D_{\min}$, so no basis of $(D*e)_{\min}$ contains $e$. Thus $e$ is a ribbon loop of $D*e$. Moreover $e$ is an orientable ribbon loop of $D*e$ because it is not a ribbon loop of $(D*e)*e=D$ and it is non-trivial because it is not a coloop of $(D*e)*e=D$.

On the other hand, if $e$ is a non-trivial orientable ribbon loop of $D*e$, then, by Definition~\ref{def:loops}(2), $e$ is not a ribbon loop of $(D*e)*e=D$. Furthermore,  $e$ is not a loop of $D\ast e$ (as it is not a trivial orientable ribbon loop), so it is not a coloop of $D$.
\end{proof}

Another illustration of how ribbon graphs can inform delta-matroids is as follows. Suppose that $G$ is a ribbon graph with a non-orientable loop $e$. If $Q$ is a maximal spanning forest of $G$ then $Q \cup e$ is a spanning quasi-tree. The following lemma shows that this property holds for delta-matroids in general.

\begin{lemma}\label{lem:unorientloops}
Let $D=(E,\mathcal{F})$ be a delta-matroid with $r(D_{\min})=r$ and suppose that $e$ is a non-orientable ribbon loop of $D$.
Then a subset $F$ of $E-e$ is a basis of $D_{\min}$ if and only if $F\cup e$ is a feasible set of $D$ with cardinality $r+1$.
\end{lemma}
\begin{proof}
Let $F \subseteq E-e$ with $|F|=r$ and $F\cup e \in \mathcal{F}$.
Suppose for contradiction that $F\notin  \mathcal{F}$. Let $A=E-(F\cup e)$.
Since $e$ is a ribbon loop of $D$, every minimum sized feasible set of $D$ contains an element of $A$.
By applying Lemma~\ref{lem:useful} we see that every feasible set $F'$ of $D$ must satisfy $|F'\cap A|\geq 1$.
However, $|(F\cup e)\cap A|=0$, a contradiction. Thus $F\in  \mathcal{F}$.

 By Definition~\ref{def:loops}(2), $e$ is non-orientable ribbon loop of $D\ast e$, and so  $r((D\ast e)_{\min})=r$. Thus, by applying the previous argument to $D\ast e$, we see that if $F\subseteq E-e$ with $|F|=r$ and $F\cup e \in \mathcal{F}(D*e)$ then $F\in \F(D*e)$. So if $F\subseteq E-e$ with $|F|=r$  and $F \in \mathcal{F}(D)$ then $F\cup e\in \F(D)$.
\end{proof}

\subsection{Deletion, contraction, and minors}\label{s5.4}

Deletion and contraction for ribbon graphs and for delta-matroids are compatible operations.
\begin{proposition}\label{p.5minor}
Let $G$ be a ribbon graph,  and $e\in E(G)$. Then
\begin{enumerate}
\item \label{p.5minor.1} $D (G\ba e)=  D(G)\ba e$;
\item  \label{p.5minor.2} $D(G/e)=D(G)/e$.
\end{enumerate}
\end{proposition}
\begin{proof}
If $e$ is a bridge of $G$ then it belongs to every spanning quasi-tree of $G$. Moreover, a subset $F$ of $E-e$ is a spanning quasi-tree of $G\ba e$ if and only if $F\cup e$ is a spanning quasi-tree of $G$. By Lemma~\ref{l.4coloop}, $e$ is a coloop of $D(G)$ and so the first part follows in this case.

On the other hand if $e$ is not a bridge of $G$, then $G$ and $G\ba e$ have the same number of connected components. Thus the spanning quasi-trees of $G\ba e$ are precisely the spanning quasi-trees of $G$ that do not contain $e$.
By Lemma~\ref{l.4coloop} again, $e$ is not a coloop of $D(G)$ and so the first part also follows in this case.

Using Proposition~\ref{p.pd2}\eqref{p.pd2.3}, we have $D(G/e)=D(G^e\ba e)$, which by Theorem~\ref{t.twdu} and the first part of this proposition is the same as $(D(G)*e)\ba e$. Using Lemma~\ref{le:delcondual1}, $(D(G)*e)\ba e = D(G)/e$.
\end{proof}

\begin{remark}
$D_{\leq k}(G\ba e) \neq (D_{\leq k}(G))\ba e$
and $D_{\leq k}(G/ e) \neq (D_{\leq k}(G))/ e$, in general. To construct an example illustrating the former, take a path with $k+1$ edges, attach a non-orientable loop to one of the vertices and let $e$ be one of the edges in the path. An example illustrating the latter can then be constructed by taking the dual. The examples with $k=1$ also illustrate that in general $D_{1}(G\ba e) \neq (D_{1}(G))\ba e$
and $D_{1}(G/ e) \neq (D_{1}(G))/ e$.
\end{remark}

The next corollary follows immediately from Proposition~\ref{p.5minor}.
\begin{corollary}\label{p.5minor3}
Let $G$ and $H$  be ribbon graphs. If $H$ is a minor of $G$ then $D(H)$ is a minor of $D(G)$.
\end{corollary}
The reverse inclusion is not true, because non-isomorphic ribbon graphs may have isomorphic ribbon-graphic delta-matroids.

We will refer to a ``$D$-minor" to mean a ``minor isomorphic to $D$".
A class $\mathcal{C}$ of delta-matroids or ribbon graphs is said to be \emph{minor-closed} if, for each $X\in \mathcal{C}$, every minor of $X$ is also in $\mathcal{C}$.
An \emph{excluded minor} for a minor-closed class $\mathcal{C}$ of delta-matroids or ribbon graphs is a delta-matroid or ribbon graph, respectively, that is not in $\mathcal{C}$ but has each of its proper minors in $\mathcal{C}$.

As a first illustration of the fact that ribbon graph intuition can lead to results about delta-matroids, we consider even delta-matroids.  Recall that an even  delta-matroid is one whose feasible sets all have the same parity. Being even is preserved under taking minors, hence it may be characterised by a set of excluded minors. Our aim is to find the set  of excluded minors for even  delta-matroids. Consider the corresponding problem for ribbon graphs. A ribbon graph is non-orientable if and only if it contains a non-orientable cycle.  Edges in a cycle can be contracted to give a loop, and it follows that a ribbon graph is orientable if an only if it has no $G_0$-minor, where $G_0$ is the ribbon graph consisting of a single non-orientable loop.
Recalling from Proposition~\ref{p.4b}\eqref{p.4b.4}, that a ribbon graph $G$ is orientable if and only if $D(G)$ is even, we deduce that
$D(G)$ is even if and only if it contains no $D(G_0)$-minor. This leads us to posit that a delta-matroid $D$ is even if and only if it contains no $X_0$-minor, where $X_0=D(G_0)=(\{a\},\{\emptyset,\{a\}\})$. This turns out to be a slight reformulation of a result of Bouchet.
\begin{theorem}[Bouchet~\cite{ab2}]
\label{exorient}
Let $X_0=(\{a\},\{\emptyset,\{a\}\})$.
A delta-matroid $D=(E,\mathcal{F})$ is even if and only if it has no $X_0$-minor.
\end{theorem}
\begin{proof}
If $D$ is even, then it clearly does not have $X_0$ as a minor, as any minor of $D$ is even.
By Bouchet's result,~\cite[Lemma~5.4]{ab2}, a delta-matroid is odd if and only if it has a feasible set $F$ and an element $e\notin F$ such that $F\cup e$ is feasible. In this case $D/F\ba(E-(F\cup e))$ is isomorphic to $X_0$, hence the result follows.
\end{proof}

\begin{remark}
As a further illustration of the interactions between ribbon graphs and delta-matroids, it is interesting to note that Bouchet's characterisation of odd delta-matroids given in the proof of Theorem~\ref{exorient} is the direct analogue of the ribbon graph result that $G$ is non-orientable if and only if it has a spanning quasi-tree $Q$ and an edge $e$ not in $Q$ such that $Q\cup e$ is a spanning quasi-tree.
\end{remark}

An excellent illustration of the compatibility between delta-matroid and ribbon graph theory is found by considering twists of matroids.
As the class of matroids is not closed under twists but every matroid is a delta-matroid, twisting provides a way to construct delta-matroids from matroids.
Delta-matroids arising from twists of matroids are of interest since they are an intermediate step between delta-matroid theory in general and the much better developed field of matroid theory. Suppose we are faced with the problem of characterising the class of delta-matroids that arise as twists of matroids. How can we use the insights of ribbon graphs to tackle this problem?

Suppose that $G=(V,E)$ is a ribbon graph with ribbon-graphic delta-matroid $D=D(G)$. We wish to understand when $D$ is the twist of a matroid, that is, we want to determine if $D=M\ast A$ for some matroid $M$ and for some $A\subseteq  E$. As twists are involutary, we can reformulate this problem as one of determining if   $D\ast B =M$ for some matroid $M$ and some $B\subseteq E$. By Theorem~\ref{t.twdu}, $D\ast B =  D(G)\ast B =   D(G^B)$, but, by Corollary~\ref{c.4c}\eqref{c.4c.2},   $D(G^B)$ is a matroid if and only if $G^B$ is a plane graph. Thus $D$ is a twist of a matroid if and only if $G$ is the partial dual of a plane graph.
Given our principle that embedded graphs inform us about delta-matroids,  to characterize the class of delta-matroids that  are twists of matroids, we should look for characterizations of the class of ribbon graphs that arise as partial duals of plane graphs. Fortunately, due to connections with knot theory (see~\cite{Mo5}), this class of ribbon graphs has been characterised.
Let $G_0$ be the ribbon graph consisting of a single non-orientable loop; $G_1$ be the orientable ribbon graph given by vertex set $\{1,2\}$, edge set $\{a,b,c\}$ with the incident edges at each vertex having the  cyclic order $abc$, with respect to some orientation of $G_1$; and let  $G_2$ be the orientable ribbon graph given by vertex set $\{1\}$, edge set $\{a,b,c\}$ with the cyclic order $abcabc$ at the vertex. Then the following holds.
\begin{theorem}[Moffatt~\cite{Moprep}]
$G$ is a partial dual of a plane graph if and only if it has no minors equivalent to $G_0$, $G_1$, or $G_2$.
\end{theorem}
The discussion above and our  principle that ribbon graphs inform us about delta-matroids lead us to the conjecture that a delta-matroid $D$ is the twist of a matroid if and only if it does not have a minor isomorphic to $D(G_0)$, $D(G_1)$, or $D(G_2)$. Indeed this result is true  and is readily  derived from work of  A. Duchamp (see~\cite{CMNR} for details of the derivation).
\begin{theorem}[Duchamp~\cite{adfund}]
\label{expdm}
A delta-matroid $D$ is the twist of a matroid if and only if it does not have a minor isomorphic to $D(G_0)$, $D(G_1)$, or $D(G_2)$.
\end{theorem}
In this example ribbon graph theory led to a result obtainable from the literature, but below we will see examples where ribbon graph theory leads to genuinely new structural delta-matroid theory.

\subsection{Separability and connectivity for delta-matroids} \label{s5.5}
If $v$ is a separating vertex of a graph $G$,  with $P$ and $Q$  being the subgraphs that intersect in $v$, then
knowledge of $P$, $Q$ and $v$ gives complete knowledge of $G$. However, if $G$ is a ribbon graph this is no longer the case. For example, suppose that $P$ and $Q$ are orientable loops. Then $G$ has genus zero or one, depending on whether or not $P$ and $Q$ are interlaced. Thus separability is a much more subtle concept for ribbon graphs than for graphs. Given our principle that graphs are matroidal, while ribbon graphs are delta-matroidal, we should expect `connectivity' for delta-matroids to be more subtle than for matroids. In this section, we define notions of connectivity and separability of delta-matroids that reflect the corresponding concepts for ribbon graphs defined in Section~\ref{sss.se}.

For matroids $M_1=(E_1,\mathcal{B}_1)$ and $M_2=(E_2,\mathcal{B}_2)$, where $E_1$ is disjoint from $E_2$, the \emph{direct sum of $M_1$ and $M_2$}, written $M_1\oplus M_2$, is constructed as follows.
\[M_1\oplus M_2:=(E_1\cup E_2,\{B_1\cup B_2\mid B_1\in \mathcal{B}_1\text{ and } B_2\in\mathcal{B}_2\}).\]
If $M=M_1\oplus M_2$, for non-trivial $M_1$ and $M_2$, then we say that $M$ is \emph{disconnected} and that $E_1$ and $E_2$ are each \emph{separating}.  We say that $M$ is \emph{connected} if it is not disconnected.
The connectivity of cycle matroids is closely linked to the connectivity of the underlying graph. A graph is \emph{$2$-connected} if it has a single connected component and no separating vertex.
The following is well-known~\cite{Oxley11}.

\begin{proposition}\label{prop:matconn}
Let $G$ be a graph. Then $M(G)$ is connected if and only if $G$ is $2$-connected. Moreover if $M(G)=M_1\oplus M_2$, for non-trivial $M_1$ and $M_2$, then $M_1=M(G_1)$ and $M_2=M(G_2)$ for some graphs $G_1=(V_1,E_1)$ and $G_2=(V_2,E_2)$ such that $G = (V_1\cup V_2,E_1\cup E_2)$, and $E_1$ and $E_2$ are disjoint, and $V_1$ and $V_2$ are either disjoint or intersect in a single vertex.
\end{proposition}

Motivated by separability for ribbon graphs, we generalize this concept to delta-matroids in two slightly different ways.  The second definition is from~\cite{geelen}.
\begin{definition}
Let $D=(E,\mathcal{F})$ be a delta-matroid. Then $D$ is \emph{separable} if $D_{\min}$ is disconnected.
\end{definition}

\begin{definition}
For delta-matroids $D=(E,\mathcal{F})$ and $\tilde{D}=(\tilde{E},\tilde{\mathcal{F}})$ with $E\cap \tilde{E}=\emptyset$, the \emph{direct sum} of $D$ and $\tilde{D}$ is written $D\oplus \tilde{D}$ and is the delta-matroid defined as
\[D\oplus \tilde{D}:=(E\cup \tilde{E},\{F\cup \tilde{F}\mid F\in \mathcal{F}\text{ and } \tilde{F}\in\tilde{\mathcal{F}}\}).\] If a delta-matroid can be written as $D\oplus \tilde{D}$ for some non-trivial delta-matroids $D$ and $\tilde{D}$, then we say it is \emph{disconnected}.  A delta-matroid is \emph{connected} if it is not disconnected.
\end{definition}

We defined separability and connectivity for delta-matroids so that they are compatible with the corresponding concepts for ribbon graphs, as in the following propositions, the first of which follows immediately from
Proposition~\ref{prop:matconn}.

\begin{proposition}\label{prop:sep}
Let $G$ be a ribbon graph. Then $D(G)$ is separable if and only if there exist non-trivial ribbon graphs $G_1$ and $G_2$ such that  $G=G_1 \sqcup G_2$ or $G = G_1 \oplus G_2$.

 Moreover if $D(G)_{\min}= M_1\oplus M_2$, for some non-trivial $M_1$ and $M_2$, then there exist non-trivial ribbon graphs $G_1$ and $G_2$ such that $D(G)_{\min} = M(G_1)\oplus M(G_2)$ and $G=G_1 \sqcup G_2$ or $G = G_1 \oplus G_2$.
\end{proposition}

\begin{proposition}
\label{coniscon}
Let $G$ be a ribbon graph. Then $D(G)$ is disconnected if and only if there exist non-trivial ribbon graphs $G_1$ and $G_2$ such that either $G=G_1 \sqcup G_2$ or $G = G_1 \curlyvee G_2$.
\end{proposition}
\begin{proof}
If $G=G_1 \sqcup G_2$ or $G = G_1 \curlyvee G_2$ then it is easy to see that $D(G)$ is disconnected.

Suppose now that $D(G)$ is disconnected. Then $D(G)$ is separable.
The previous proposition implies
that this is only possible if
$G=G_1 \sqcup G_2$ or $G = G_1 \oplus G_2$ for some non-trivial $G_1$ and $G_2$. Moreover if $D(G)=D \oplus D'$ then $D_{\min}=D(G_1)_{\min}$ and $D'_{\min}=D(G_2)_{\min}$ for non-trivial ribbon graphs $G_1$ and $G_2$ such that
$G=G_1 \sqcup G_2$ or $G = G_1 \oplus G_2$.

It remains to show that if $G=G_1\oplus G_2$, but $G \ne G_1\curlyvee G_2$ then
$D(G) \ne D(G_1) \oplus D(G_2)$.
If $G=G_1\oplus G_2$, but $G \ne G_1\curlyvee G_2$ then there are two interlaced cycles $C_1$ and $C_2$ of $G_1$ and $G_2$, respectively,
intersecting in $G$ at a vertex $v$.
 Let $e_1\in E(C_1)$ and let $F_1$ be a maximal forest of $G_1$ with $C_1-\{e_1\}\subseteq F_1$. Define $F_2$ similarly. Now $F_i\cup \{e_i\} \in \mathcal{F} (D(G_i))$ if and only if $C_i$ is non-orientable. However $(F_1  \cup \{e_1\})\cup (F_2\cup \{e_2\}) \in \mathcal{F}(D(G))$ except when both $C_1$ and $C_2$ are non-orientable. Consequently $D(G) \ne D(G_1) \oplus D(G_2)$.
\end{proof}

We emphasize the unfortunate clash between  ribbon graph and delta-matroid notation  that while $D(G_1\curlyvee G_2) = D(G_1)\oplus D(G_2)$, in general, $D(G_1\oplus  G_2) \neq D(G_1)\oplus D(G_2)$.

For another illustration of how ribbon graphs inform delta-matroids we return to the problem of characterising twists of matroids from the end of Section~\ref{s5.4}. In that section we saw how ribbon graph theory led to an excluded minor characterisation of twists of matroids. We will now see how they lead to a rough structure theorem for twists of matroids.

As before the ribbon graph analogue of a twist of a matroid is a partial dual of a plane graph. Motivated by knot theory, in~\cite{Mo5} (see also~\cite{Mof11c}), Moffatt gave a rough structure theorem for the class of partial duals of plane graphs. This rough structure theorem ensures that every such ribbon graph admits a particular decomposition into plane ribbon graphs.
\begin{theorem}[Moffatt~\cite{Mo5}]\label{rgppd}
Let $G$ be a ribbon graph and $A\subseteq E(G)$. Then the partial dual $G^A$ is a plane graph if and only if  all of the connected components of $G|_A$ and $G|_{A^c}$ are plane and
 every vertex of $G$ that is in both $G|_A$ and $G|_{A^c}$ is a separating vertex of $G$.
\end{theorem}
We now translate this into delta-matroids.  If $D=D(G)$ then ``$G^A$ is a plane graph'' becomes ``$D\ast A$ is a matroid'', and ``$G|_A$ and $G|_{A^c}$ are plane'' becomes ``$D\setminus A^c$ and $D\setminus A$ are both matroids''. By Proposition~\ref{prop:sep}, $D(G)$ is separable if and only if there exist ribbon graphs $G_1$ and $G_2$ such that  $G=G_1 \sqcup G_2$ or $G = G_1 \oplus G_2$. Thus the condition that every vertex of $G$ that is incident with edges in $A$ and edges in $E(G)-A$ is a separating vertex of $G$ becomes $A$ is separating in $D_{\min}$.
Thus we have deduced the following theorem for ribbon-graphic delta-matroids. Our principle  that ribbon graphs inform us about delta-matroids led us to conjecture that it holds for delta-matroids in general, and we showed that this is indeed the case.
\begin{theorem}[Chun et al~\cite{CMNR}]
\label{charaterise_delta_matroids_twists_matroids}
Let $D$ be a delta-matroid and $A$ be a non-empty proper subset of $E(D)$. Then $D*A$ is a matroid if and only if the following two conditions hold:
\begin{enumerate}
\item $A$ is separating in $D_{\min}$, and
\item $D\setminus A$ and $D\setminus A^c$ are both matroids.
\end{enumerate}
\end{theorem}
We emphasise that the  ribbon graph theory genuinely led us to the formulation of Theorem~\ref{charaterise_delta_matroids_twists_matroids}. We  probably would not have found the result without the insights and guidance of ribbon graphs.

We have just given an example of how ribbon graphs inform delta-matroids. We now give an example of delta-matroid theory giving a result about ribbon graphs.

The following inductive tools have been fundamental in the development of matroid theory.

\begin{theorem}[Tutte~\cite{Tutte66}]
\label{1chain}
Let $M$ be a connected matroid.
If $e\in E(M)$, then $M\backslash e$ or $M/e$ is connected.
\end{theorem}

\begin{theorem}[Brylawski~\cite{Bry72}, Seymour~\cite{Seymour77}]
\label{lsplitter}
Let $M$ be a connected matroid with a connected minor $N$.
If $e\in E(M)-E(N)$, then $M\backslash e$ or $M/e$ is connected with $N$ as a minor.
\end{theorem}

Bouchet generalized Theorem~\ref{1chain} to the context of delta-matroids in~\cite{mmiii}.
The actual result that he proved is for an even more general object, called a multimatroid, but we state a special case of his result here in terms of delta-matroids.
\begin{theorem}[Bouchet~\cite{mmiii}]
\label{chainmm}
Let $D$ be a connected even delta-matroid.
If $e\in E(D)$, then $D\ba e$ or $D/e$ is connected.
\end{theorem}

By exploiting results of Brijder and Hoogeboom~\cite{BH14}, Chun, Chun, and Noble in~\cite{chch} derived another consequence of Bouchet's result,
extending Theorem~\ref{chainmm} to the class of vf-safe delta-matroids. These were introduced by Brijder and Hoogeboom in~\cite{BH13}, and include ribbon-graphic delta-matroids and binary delta-matroids, which are discussed in Section~\ref{s5.7}.
Given a vf-safe delta-matroid $D$ and a subset $A$ of its ground set, the delta-matroid $D+A$ has ground set $E(D)$ and $F$ is feasible if and only if $D$ has an odd number of feasible sets $F'$ satisfying $F-A \subseteq F' \subseteq F$.
For more details on the delta-matroid $D+A$, including how the $+$ operation interacts with other delta-matroid operations such as twisting, see~\cite{BH11,BH13}.
\begin{theorem}[Chun et al.~\cite{chch}]\label{chainvf}
Let $D$ be a connected vf-safe delta-matroid. If $e\in E(D)$, then at least two of $D\ba e$, $D/e$ and $(D+e)/e$ are connected.
\end{theorem}

For a ribbon graph $G$ and subset $A$ of its edges, informally we define $G+A$ to be the ribbon graph formed from $G$ by adding a ``half-twist'' to the edges in $A$ (see~\cite{EMM} for a formal definition of the partial Petrial and Petrie dual). It is shown in~\cite{CMNR} that $D(G+A)=D(G)+A$. A ribbon graph is said to be \emph{$2$-connected} if
$G\neq P\sqcup Q$ and $G\neq P\curlyvee Q$, for any non-trivial ribbon graphs $P$ and $Q$.
The next theorem, also proved by Chun, Chun, and Noble~\cite{chch} follows immediately from the two preceding theorems and Proposition~\ref{coniscon}.

\begin{theorem}[Chun et al.~\cite{chch}]\label{th:chainrib}
Let $G$ be a 2-connected ribbon graph. Then at least two of $G\ba e$, $G/e$ and $(G+e)/e$ are $2$-connected.
\end{theorem}

In~\cite{chch}, Chun, Chun, and Noble generalized Theorem~\ref{lsplitter} to multimatroids. We state two special cases of the result here in terms of delta-matroids.
\begin{theorem}[Chun et al.~\cite{chch}]
\label{splitter}
Let $D$ be a connected even delta-matroid with a connected minor $D'$.
If $e\in E(D)-E(D')$, then $D\backslash e$ or $D/e$ is connected with $D'$ as a minor.
\end{theorem}
To state the second special case, we need to concept of a $3$-minor in a vf-safe delta-matroid. We say that $D'$ is a \emph{$3$-minor} of a vf-safe delta-matroid $D$, if $D' = ((D \ba X / Y) + Z) / Z$ for disjoint subsets $X$, $Y$ and $Z$ of $E(D)$. It is not difficult to establish that the three operations used in forming a $3$-minor have the desirable property that they may be applied element by element in any order without changing the result.
\begin{theorem}[Chun et al.~\cite{chch}]
\label{splittervf}
Let $D$ be a connected vf-safe delta-matroid with a connected $3$-minor $D'$.
If $e\in E(D)-E(D')$, then $D\backslash e$, $D/e$ or $(D+e)/e$ is connected with $D'$ as a $3$-minor.
\end{theorem}

The next two results follow immediately from the previous two.

\begin{theorem}[Chun et al.~\cite{chch}]\label{th:splittergrapheven}
Let $G$ be a $2$-connected, orientable ribbon graph. If $H$ is a $2$-connected minor of $G$ and $e \in E(G)-E(H)$, then $G\ba e$ or $G/e$ is $2$-connected with $H$ as a minor.
\end{theorem}

A $3$-minor in a ribbon graph in an analogous way to which it is defined in a vf-safe delta-matroid.

\begin{theorem}[Chun et al.~\cite{chch}]\label{th:splittergraph}
Let $G$ be a $2$-connected ribbon graph. If $H$ is a $2$-connected $3$-minor of $G$ and $e \in E(G)-E(H)$, then $G\ba e$, $G/e$ or $(G+e)/e$ is $2$-connected with $H$ as a $3$-minor.
\end{theorem}

As we mentioned above, this result is a nice example of delta-matroids providing insight into ribbon-graphs. It is extremely unlikely that we would have established Theorem~\ref{th:splittergraph} without the intuition provided by delta-matroids.

\subsection{Rank functions}\label{s5.6}
In this section we examine delta-matroid rank and its connections to ribbon graph structures.
Let $G=(V,E)$ be a  graph, $M=M(G)$ be its cycle matroid, and $A\subseteq E$. It is well-known that the rank function of $M$ can be expressed in terms of graph parameters: $r_M(A)= v(G)-k_G(A)$. In this section we  express the rank function of a ribbon-graphic delta-matroid in terms of ribbon graph parameters.

For our next proof, we need a new piece of terminology.
Let $H$ and $K$ be distinct spanning ribbon subgraphs of $G$. Then we say that $K$ is obtained from $H$ by an \emph{edge-toggle} if   $E(H)=E(K)\btu e$ for some edge $e\in E(G)$.
Recall that, for ribbon graph $G$ and $A\subseteq E(G)$, functions such as $\rho_{D(G)}(A)$, $e(A)$, and $f(A)$ refer to $\rho_{D(G)}((V(G),A))$, $e((V(G),A))$, and $f((V(G),A))$, respectively, as defined in Section~\ref{sss.srs}.

\begin{theorem}\label{t.rankinterp}
Let $G=(V,E)$ be a ribbon graph and $A\subseteq E$. Then
\[ \rho_{D(G)}(A) =  e(G)-f(A)+k(G) . \]
\end{theorem}
\begin{proof}
To prove the theorem it is enough to show that for a ribbon graph $G=(V,E)$ with $D(G)=(E,\mathcal{F})$ we have
$\min\{|A\bigtriangleup F| \mid F\in \mathcal{F}\} = f(A)-k(G)$.
To do this, set
\[q(A):= \min\{  |X|+|Y| \mid X,Y \subseteq E, \text{ and }  (V, (A-X)\cup Y)  \text{ is a spanning quasi-tree} \}. \]
Then $q(A)$ is the smallest number of edge-toggles needed to transform $(V,A)$ into a spanning quasi-tree. Clearly $q(A)=\min\{|A\bigtriangleup F| \mid F\in \mathcal{F}\}$, and so we need to show that $q(A)=f(A)-k(G)$.

First observe that $q(A)\geq f(A)-k(G)$ since an edge-toggle  can decrease the number of boundary components by at most one.

To show that $q(A)\leq f(A)-k(G)$ we argue by induction on $f(A)$.
If $f(A)=k(G)$, then $(V,A)$ is a spanning quasi-tree and $q(A)=0=f(A)-k(G)$.
For the inductive hypothesis, suppose that $q(A)\leq f(A)-k(G)$ for all $A$ with $f(A)< r$.
Now suppose that $f(A)=r>k(G)$. There are two cases to consider: $k(A)>k(G)$ and $k(A)=k(G)$.

If $k(A)>k(G)$, then $G$ has an edge $e\notin A$ such that $k(A\cup e)=k(A)-1$. Then we must also have $f(A\cup e)=f(A)-1$. The inductive hypothesis then gives
$q(A\cup e) \leq f(A\cup e) -k(G) =  f(A)-k(G)-1$. So a sequence  of at most $f(A)-k(G)-1$ edge-toggles transforms $(V,A\cup e)$ to a spanning quasi-tree. Placing `add $e$' at the start of this sequence of edge-toggles  gives a sequence of at most $f(A)-k(G)$ edge-toggles that transforms $(V,A)$ to a spanning quasi-tree. Thus  $q(A) \leq f(A)-k(G)$.

If $k(A)=k(G)$, then, since $(V,A)$ has more than $k(G)$ boundary components, $A\neq \emptyset$. Each edge of $(V,A)$ intersects either one or two boundary components of $(V,A)$. There must be some edge $e\in A$ that intersects two boundary components since $f(A)>k(G)$ and $k(A)=k(G)$.  Then $f(A- e)=f(A)-1$. The inductive hypothesis then gives
$q(A- e) \leq f(A- e) -k(G) =  f(A)-k(G)-1$. So, proceeding as in the case where $k(A)>k(G)$,  a sequence  of at most $f(A)-k(G)-1$ edge-toggles transforms $(V,A- e)$ to a spanning quasi-tree. Placing  `subtract $e$' at the start of this sequence of edge-toggles gives a sequence of at most $f(A)-k(G)$ edge-toggles that transforms $(V,A)$ to a spanning quasi-tree. Thus  $q(A) \leq f(A)-k(G)$. This completes the proof of the theorem.
\end{proof}

\begin{remark}
This theorem can be seen as a corollary of the extended Cohn-Lempel equality from~\cite{TRA11}. One associates to the ribbon graph its medial graph, which is 4-regular. Next, one translates the ribbon graph parameters $e(G), f(A)$ and $k(G)$ into parameters depending on the medial graph. One can also construct the delta-matroid of the ribbon graph from the medial graph, see~\cite{ab2} and the proof of Theorem~\ref{ribboneuler}. Once the definition of the rank of the delta-matroid is translated to the medial graph, Theorem~\ref{t.rankinterp} follows from the extended Cohn-Lempel equality.
\end{remark}

The theorem above immediately provides us with the following interpretation of $\rho$ for ribbon-graphic delta-matroids.
\begin{corollary}
Let $G=(V,E)$ be a ribbon graph, $A\subseteq E(G)$, and $D=D(G)$. Then $|E|-\rho_D(A)$ is equal to the minimum number of edge-toggles required to transform $(V,A)$ into a spanning quasi-tree of $G$.
\end{corollary}

The ribbon graph interpretation of $\rho_{D(G)}$ can be used to discover results about $\rho_D$ for a general delta-matroid $D$. For example, recall from the proof of Lemma~\ref{l.3a} that the boundary components of $G\ba A^c$ and $G^*\ba A$ coincide and so $f_G(A)=f_{G^*}(A^c)$. Thus, for ribbon-graphic delta-matroids, it follows that $\rho_{D^*} (A) = \rho_{D} (E-A)$. This identity holds  for delta-matroids in general, as we saw earlier in Lemma~\ref{le:rankdual}.

For reference later, we record the following basic facts about rank functions.
\begin{corollary}\label{c.ranks}
Let $G=(V,E)$ be a ribbon graph. Then
\begin{enumerate}
\item $r_{M(G)}(A)   = r_{D(G)_{\min}}(A)$;
\item \label{c.ranks.2} $ r_{M(G^*)}(A)  =   r_{(D(G)_{\max})^*}(A)=   r_{(D(G)_{\max})}(A^c)+|A|-r_{(D(G)_{\max})}(E)$;
\item $  \rho_{D(G)} (A)=  \rho_{D(G^*)}(E- A)$.
\end{enumerate}
\end{corollary}
\begin{proof}
The first part follows immediately from the fact that $M(G)=D(G)_{\min}$.
For the second part, first note that Corollary~\ref{c.4c}\eqref{c.4c.4} implies that
 $M(G^*)=(D(G)_{\max})^*$. Thus
$r_{M(G^*)}(A)  =   r_{(D(G)_{\max})^*}(A)$.  Equation~\eqref{eq:matrank} implies that this is equal to $r_{(D(G)_{\max})}(A^c)+|A|-r_{(D(G)_{\max})}(E)$.
Thus \eqref{c.ranks.2} holds.
As $(D(G))^*=D(G^*)$ by Corollary~\ref{c.gedu}, the third part follows from Lemma~\ref{le:rankdual}.
\end{proof}

To motivate some delta-matroid results, consider a ribbon graph $G$ and a set $A\subseteq E(G)$. Then
$r(A)=v(G)-k(A)$ and $\rho(A)=e(G)-f(A)+k(G)$. Euler's formula and Proposition~\ref{p.4b}\eqref{p.4b.3} give
 \begin{align*}
\rho (A) -r(A)- n(G)+n(A) &=  (e(G)-f(A)+k(G)) - (v(G)-k(A))\\ &\phantom{=} \mbox{ } - (e(G) -v(G)+k(G)) + n(A))\\
& =   k(A) - f(A) + n(A) = \gamma(A)\\ & =\spn(D(G \ba A^c)) = \spn(D(G)\ba A^c)=\spn(D(G)|A) .  \end{align*}

This identity holds more generally for delta-matroids, which we will show after we state the following lemma, the simple proof of which we omit.
\begin{lemma}\label{le:simprank}
Let $D=(E,\mathcal{F})$ be a delta-matroid. Then $r(D_{\max})= \rho_D(E)$ and $r(D_{\min})=|E|-\rho_D(\emptyset)$.
\end{lemma}

\begin{proposition}\label{p.rasp}
Let $D=(E,\mathcal{F})$ be a delta-matroid and let $A\subseteq E$. Then
\begin{enumerate}
\item \label{p.rasp.1} $r((D|A)_{\min}) = r_{D_{\min}}(A)$;
\item $r((D|A)_{\max}) = \rho_D (A) - n_{D_{\min}}(E)+n_{D_{\min}}(A)$;
\item $\spn(D|A)= \rho_D (A) -r_{D_{\min}}(A)- n_{D_{\min}}(E)+n_{D_{\min}}(A)$.
\end{enumerate}
\end{proposition}

\begin{proof}
Let $F_0$ be a feasible set of $D$ having smallest possible intersection with $A^c$. By Lemma~\ref{lem:useful}, we may assume that $F_0 \in \mathcal{F}(D_{\min})$.
Let $Y=F_0 \cap A^c$ and $Z=A^c-Y$. If the elements of $Z$ are deleted one by one from $D$, then no coloop is deleted because there is a feasible set $F_0$ missing $Z$. However every element of $Y$ is a coloop of $D\ba Z$. Thus $D|A=D\ba Z/Y$. We have
\[ \mathcal{F}(D|A) = \{F-Y \mid F\in\mathcal{F(D)},\ F\cap A^c=Y\}.\]
Therefore
\begin{equation} r((D|A)_{\min}) = |F_0|-|Y| = \max_{F\in\mathcal{F}(D_{\min})}\{|F \cap A|\} = r_{D_{\min}}(A),\label{eq:minrestrict}\end{equation}
establishing the first part.

Applying Lemma~\ref{minorbirank} $|A^c|$ times to delete first the elements of $Z$ and then those of $Y$ implies that
\begin{equation} \rho_{D|A}(A) = \rho_D(A) - |E|+|A|+|Y|.\label{eq:rhodelminor}\end{equation}

By applying Lemma~\ref{le:simprank} to $D|A$ and Equation~\eqref{eq:rhodelminor}, we obtain
\begin{equation} r((D|A)_{\max}) = \rho_{D|A}(A) = \rho_D(A) + |A| + |Y| - |E|.\label{eq:maxrestrict}\end{equation}
Now $n_{D_{\min}}(E)=|E|-|F_0|$ and, by Equation~\eqref{eq:minrestrict}, $n_{D_{\min}}(A)=|A|-(|F_0|-|Y|)$. Substituting into Equation~\eqref{eq:maxrestrict} yields the second part.

The final part follows immediately by subtracting the equation in the first part from that in the second part.
\end{proof}

\subsection{Representability}\label{s5.7}
Let $\mathbb{K}$ be a finite field. For a finite set $E$, let $C$ be a skew-symmetric $|E|$ by $|E|$ matrix over $\mathbb{K}$, with rows and columns indexed, in the same order, by the elements of $E$. Note that we only allow the diagonal of $C$ to be non-zero when $\mathbb{K}$ has characteristic two.
Let $C\left[ A\right]$ be the principal submatrix of $C$ induced by the set $A\subseteq E$.

We define the delta-matroid $D(C)=(E,\F)$, where  $A\in \F$ if and only if $C[A]$ is non-singular over $\mathbb{K}$. By convention $C[\emptyset]$ is non-singular.
Bouchet  showed in~\cite{abrep} that $D(C)$ is indeed a delta-matroid. Observe that $\emptyset\in \F(D(C))$, for every  $C$.

A delta-matroid is called \emph{representable over $\mathbb{K}$} if it has a twist that is isomorphic to $D(C)$ for some matrix $C$.

\begin{lemma}[Bouchet~\cite{abrep}]\label{for1}
Suppose that a delta-matroid $D$ is representable over a field $\mathbb{K}$. Let $F$ be any feasible set of $D$. Then $D\ast F=D(C)$ for some  skew-symmetric   matrix $C$ over $\mathbb{K}$.
\end{lemma}

Suppose that $M$ is a matroid representable over $\mathbb{K}$ and that $B$ is a basis of $M$. Then $M$ has a representation of the form $(I|A)$ where $I$ is a $|B|$ by $|B|$ identity matrix and the columns of $I$ correspond to the elements of $B$. It is not difficult to see that if
\[ C = \begin{pmatrix} 0 & A \\ -A^T & 0\end{pmatrix},\]
then $M*B = D(C)$. Thus we have the following result.
\begin{proposition}[Bouchet~\cite{abrep}]\label{prop:binisbin}
A matroid representable over a field $\mathbb{K}$ is also representable over $\mathbb{K}$ as a delta-matroid.
\end{proposition}

A delta-matroid representable over the field with two elements is called \emph{binary}.
If $D=D(C)$ is a binary delta-matroid, then its feasible sets of all sizes are  determined by its feasible sets of size at most two. By combining this observation with  Lemma~\ref{for1}, we obtain the following.
\begin{lemma}[Bouchet and Duchamp~\cite{BD91}]\label{for2}
Let $F$ be a feasible set of a binary delta-matroid $D$. Then the feasible sets of $D$ are determined by
$\{X \mid |F\btu X| \leq 2\text{ and } X\in \F(D) \}$.
\end{lemma}

Bouchet and Duchamp gave an excluded-minor characterisation of binary delta-matroids.
\begin{theorem}[Bouchet and Duchamp~\cite{BD91}]
\label{binarychar}
A delta-matroid is a binary delta-matroid if and only if it has no minor isomorphic to a twist of $S_1, S_2,S_3,S_4$, or $S_5$, where
\begin{enumerate}
\item $S_1=(\{1,2,3\},\{\emptyset , \{1,2\},\{1,3\}, \{2,3\},\{1,2,3\}\})$,
\item $S_2=(\{1,2,3\},\{\emptyset , \{1\},\{2\}, \{3\}, \{1,2\},\{1,3\}, \{2,3\}\})$,
\item $S_3=(\{1,2,3\},\{\emptyset ,\{2\}, \{3\}, \{1,2\},\{1,3\},\{1,2,3\}\})$,
\item $S_4=(\{1,2,3,4\},\{\emptyset , \{1,2\},\{1,3\}, \{1,4\},\{2,3\}, \{2,4\},\{3,4\}\})$,
\item $S_5=(\{1,2,3,4\},\{\emptyset , \{1,2\},\{1,4\},\{2,3\}, \{3,4\},\{1,2,3,4\}\})$.
\end{enumerate}
\end{theorem}

It is easy to check that no twist of $S_1$, $S_2$, $S_3$ or $S_4$ is a matroid and that the uniform matroid $U_{2,4}$ is the only twist of $S_5$ that is a matroid.
Note that this result implies Tutte's characterization of
binary matroids~\cite{Tutte56} because $U_{2,4}$ is the unique excluded minor for the class of binary matroids.

\medskip

It is well known that  graphic matroids are representable over every field. An analogous result holds for ribbon graphic delta-matroids.
Let $D$ be a ribbon-graphic delta-matroid. It is readily verified that $S_1, \ldots, S_5$ do not arise as the delta-matroids of any ribbon graph.
Consequently $D$ has no twist of any delta-matroid in $\{S_1,\dots ,S_5\}$ as a minor. So Theorem~\ref{binarychar} implies that $D$ is binary.
\begin{theorem}[Bouchet~\cite{abrep}]
\label{ribbonisbinary}
Every ribbon-graphic delta-matroid is a binary delta-matroid.
\end{theorem}

Knowing that $D(G)$ is binary, it is straightforward to write down a binary representation for the delta-matroid of a ribbon graph $G=(V,E)$ that has a single vertex.
Let $C=(c_{e,f}\mid e,f\in E)$ be the binary matrix representing $D(G)$.
Let $c_{e,e}$ be one if $e$ is non-orientable and let $c_{e,e}$ be zero otherwise.
Let both $c_{e,f}$ and $c_{f,e}$ be one if $e$ and $f$ are interlaced; otherwise they are both zero.

If $G$ is connected and has more than one vertex, then a binary representation for $D(G)$ can be found by forming the partial dual $G^Q$, where $Q$ is the edge set of a spanning quasi-tree, then forming a matrix $C$ as above using $G^Q$.

Bouchet's proof of Theorem~\ref{ribbonisbinary}  predates Theorem~\ref{binarychar}, and is more involved. The difficulty is showing that $D(G)=D(C)$.
He extended Theorem~\ref{ribbonisbinary} to other fields as follows.
\begin{theorem}[Bouchet~\cite{abrep}]
An even ribbon-graphic delta-matroid is representable over any field.
\end{theorem}

As even ribbon-graphic delta-matroids correspond precisely to the delta-matroids formed from orientable ribbon graphs, the following is obvious.

\begin{corollary}[Bouchet~\cite{abrep}]
The delta-matroids of orientable ribbon graphs are representable over any field.
\end{corollary}

\begin{remark}
If $\mathbb{K}$ is a field with a characteristic different from two, any nonsingular skew-symmetric matrix is of even size. Hence any delta-matroid that is representable over a field of charactistic different from two has to be even. Thus the delta-matroid of any non-orientable ribbon graph is not representable over any field with characteristic different from two.
\end{remark}

\begin{remark}
Not all binary delta-matroids are ribbon-graphic. The matroid $M(K_5)$ is a binary matroid and hence by Proposition~\ref{prop:binisbin} it is a binary delta-matroid. However, it is not ribbon-graphic.
If $M(K_5)$ is isomorphic to $D(G)$ for some graph $G$, then by Corollary~\ref{c.4c}\eqref{c.4c.2}, $G$ must be planar, and then $D(G)$ and $M(G)$ are isomorphic. This is impossible because
$M(K_5)$ is not isomorphic to the cycle matroid of any other graph, and $G$ is planar but $K_5$ is not.
\end{remark}

\subsection{Characterising ribbon-graphic delta-matroids}\label{s5.8}

Just as not all matroids are graphic, not all delta-matroids are ribbon-graphic. It is natural to ask for a characterisation of ribbon-graphic delta-matroids, and such a characterisation can be recovered from work of Geelen and Oum.  In~\cite{oumg} Geelen and Oum built on the work of Bouchet~\cite{abcircle} in the area of circle graphs and found pivot-minor-minimal non-circle-graphs. As an application of this they obtained the excluded minors for ribbon-graphic delta-matroids.
\begin{theorem}[Geelen and Oum~\cite{oumg}]
\label{exribbon}
A delta-matroid is ribbon-graphic if and only if it does not contain a minor isomorphic to a twist of a delta-matroid in $\{S_1,S_2,\dots ,S_5\}$, where $S_1,S_2,\dots ,S_5$ are as in Theorem~\ref{binarychar}, or in the set of 166 binary delta-matroids found by the authors of~\cite{oumg}.
\end{theorem}

\section{Topological analogues of the Tutte polynomial}
\label{s6}

The \emph{Tutte polynomial}, $T(G;x,y)$, of a graph or ribbon graph $G=(V,E)$  can be defined as the state sum
\begin{equation*}
T(G;x,y) =\sum\limits_{A\subseteq E(G)} (x-1)^{r(G)-r(A)}(y-1)^{n(A)}.
\end{equation*}
The Tutte polynomial is perhaps the most studied of all graph polynomials because of the vast range of its specializations, including graph invariants from statistical physics and knot theory, and because of its interplay with other key graph polynomials such as the interlace polynomial, Penrose polynomial, chromatic polynomial and flow polynomial. Tutte introduced his eponymous polynomial in~\cite{Tutte47}. A good recent survey is~\cite{EMCM}. More details on specializations can be found in~\cite{Webook} and~\cite{BryOx}, and historical background can be found in~\cite{Farr}.

We think of the Tutte polynomial as a polynomial over the ring of integers, $T(G;x,y) \in \mathbb{Z}[x,y]$. Both it and all the other polynomials in this section can also be defined over an arbitrary commutative unitary ring, but, for simplicity of exposition, we will work over $\mathbb{Z}$.

It is well-known that the Tutte polynomial is matroidal, in the sense that all of its parameters depend only on the cycle matroid $M(G)$ of $G$, rather than the graph itself. It is defined for all matroids by replacing $G$ with $M$ in the definition above. The Tutte polynomial can readily be extended to delta-matroids by setting
\begin{equation*}
T(D;x,y) := T(D_{\min};x,y)=   \sum\limits_{A\subseteq E(D)} (x-1)^{r_{D_{\min}}(D)-r_{D_{\min}}(A)}(y-1)^{n_{D_{\min}}(A)}.
\end{equation*}
Since $D(G)_{\min}=M(G)$, we have $T(D(G);x,y)=T(G;x,y)$.

There has been much recent interest in extensions of the Tutte polynomial to embedded graphs and ribbon graphs. By the term `extension' here we mean that the polynomial should include the Tutte polynomial as a specialization, and that it should encode topological information about the embedding of the graph in some way. We refer to such polynomials loosely as `topological Tutte polynomials'. The Tutte polynomial itself clearly does not depend upon the embedding.

Here we are concerned with three such polynomials: the Las~Vergnas polynomial, the ribbon graph polynomial of Bollob\'as and Riordan, and the Krushkal polynomial. We show that, while the Tutte polynomial is matroidal, the topological Tutte polynomials are delta-matroidal, that is, they depend only on the delta-matroid of a ribbon graph, and they are well-defined for delta-matroids.

Why should we expect this to be the case? Above we defined the Tutte polynomial in terms of a sum over spanning subgraphs of $G$. The Tutte polynomial was originally defined (see~\cite{Tutte47}) as a sum over the set of maximal spanning forests of $G$.  It was recently shown that  each of the three topological Tutte polynomials mentioned above can be expressed as a sum over the set of spanning quasi-trees of a ribbon graph.  See~\cite{CKS07,Dew,VT10} for the ribbon graph polynomial, and~\cite{Bu12} for the Krushkal and Las~Vergnas polynomials.
Given that $T(G)$ is determined by $M(G)$, which is in turn determined by the set of maximal spanning forests of $G$, and the topological Tutte polynomials are determined by their spanning quasi-trees which also determine $D(G)$,
it seems reasonable to expect, and it is indeed the case, that the topological Tutte polynomials are determined by $D(G)$.

We consider the three polynomials in their chronological order, and so start with the Las Vergnas polynomial $L(G;x,y,z)$ from~\cite{Las78,Las80,Las78a}. The Las Vergnas  polynomial arose as a special case of Las~Vergnas' Tutte polynomial of a morphism of  matroids of~\cite{Las75}, and can be defined in terms of the cycle  matroid $M(G)$ of an embedded graph $G$ and the \emph{bond matroid} $B(G^*):=(M(G^*))^*$ of its geometric dual $G^*$.
  The \emph{Las~Vergnas polynomial}, $L(G;x,y,z)\in \mathbb{Z}[x,y,z]$, of an embedded graph or ribbon graph $G$ is defined by
  \begin{multline*}
   L(G;x,y,z):= \sum_{A \subseteq E( G)}(x - 1)^{ r_{M(G)}(E) - r_{M(G)}( A )}\\
               \cdot  (y-1)^{ n_{B(G^*)}(A)}
                z^{r_{B(G^*)}(E) -r_{M(G)}(E)-(r_{B(G^*)}(A)-r_{M(G)}(A))}.
  \end{multline*}
Observe that when $G$ is a plane graph, then   $B(G^*)= (M(G^*))^*=M(G)$ and so $L(G;x,y,z) =T(G;x,y)$.
Las Vergnas~\cite{Las80} proved that for any embedded graph $G$, \[(y-1)^{\gamma(G)} L(G; x,y,  1/(y-1) ) =T(G;x,y).\]

Recalling from Corollary~\ref{c.4c}\eqref{c.4c.4} that $D(G)_{\min}=M(G)$ and $D(G)_{\max}= (M(G^*))^*= B(G^*)$, it is clear how to extend   $L(G;x,y,z)$  to delta-matroids.
  \begin{definition}
Let $D=(E,\mathcal{F})$ be a delta-matroid. Then the Las~Vergnas polynomial $L(D;x,y,z)$ is given by
\begin{multline*}  L(D;x,y,z) := \sum_{A \subseteq E}    (x - 1)^{ r_{D_{\min}}(E) - r_{D_{\min}}( A )}\\
\cdot                 (y-1)^{ n_{D_{\max}}(A)}
                z^{r_{D_{\max}}(E) -r_{D_{\min}}(E)-(r_{D_{\max}}(A)-r_{D_{\min}}(A))}.\end{multline*}
                \end{definition}

It is immediate from the definition  that the ribbon graph and delta-matroid versions of $L(G)$ coincide.
\begin{theorem}
Let $G$ be a connected ribbon graph. Then
\[ L(G;x,y,z)= L(D(G);x,y,z). \]
\end{theorem}

Just as with the ribbon graph version, $L(D;x,y,z) =T(D;x,y)$ when $D$ is a matroid, and  for any delta-matroid $D$ we have \[(y-1)^{\spn(D)} L(D; x,y,  1/(y-1) ) =T(D;x,y).\] To see why this identity holds, expand and simplify the exponents of $(y-1)^{\spn(D)} L(D; x,y,  1/(y-1) )$, noting that $\spn(D)=r_{D_{\max}}(E) -r_{D_{\min}}(E)$.

\medskip

The chronologically second and  most studied of the three topological graph polynomials in this section is Bollob\'as and Riordan's ribbon graph polynomial of~\cite{BR1,BR2}. Let $G=(V,E)$ be a ribbon graph. Then the \emph{ribbon graph polynomial} or the \emph{Bollob\'as-Riordan polynomial} of $G$, denoted by $R(G;x,y,z,w) \in \mathbb{Z}[x,y,z,w]/ \langle w^2  - w\rangle$,    is defined by
\begin{equation}\label{e.defR}
R(G;x,y,z,w) = \sum_{A \subseteq E}   (x - 1)^{r( E ) - r( A )}   y^{n(A)} z^{\gamma(A)} w^{t(A)} .
\end{equation}
To extend this polynomial to  delta-matroids $D=(E,\mathcal{F})$, first, for $A\subseteq E$, define $t(A)$ by setting $t(A)=0$  if $D|A$ is even, and $t(A)=1$ otherwise. Next observe that, by Lemma~\ref{p.4b}\eqref{p.4b.3} and Proposition~\ref{p.5minor}, we have $\gamma(A)=\gamma(G\ba A^c)= \spn(D(G\ba A^c))= \spn(D(G) | A)$. To simplify notation a little, we let $\spn_D(A):= \spn (D|A)$.
\begin{definition}
Let $D=(E,\mathcal{F})$ be a delta-matroid. Then the \emph{Bollob\'as-Riordan polynomial} $R(D;x,y,z,w) \in \mathbb{Z}[x,y,z,w]/
                  \langle w^2  - w\rangle$, of $D$ is
                  \[   R(D;x,y,z,w) := \sum_{A \subseteq E}   (x - 1)^{r_{D_{\min}}( E ) - r_{D_{\min}}( A )}   y^{n_{D_{\min}}(A)} z^{ \spn_D(A)}   w^{t(A)} .\]
\end{definition}

By construction, the ribbon graph and delta-matroid versions of $R(G)$ coincide, that is, Bollob\'as and Riordan's ribbon graph polynomial is delta-matroidal.
\begin{theorem}\label{t.brdet}
Let $G$ be a  ribbon graph. Then \[R(G;x,y,z,w)= R(D(G);x,y,z,w).\]
\end{theorem}
Recall from Section~\ref{s4.2} that the isotropic matroid of a ribbon graph $G$ is defined in terms of $G$ and a quasi-tree $Q$ of $G$. In~\cite{Tr15trans}, Traldi, working in the language of transition matroids, showed that $R(G)$ can be determined from $k(G)$, the isotropic matroid of $G$ and the quasi-tree $Q$.
By the discussion following Corollary~\ref{cor:beforetraldi},  the isotropic matroid and a quasi-tree determine $D(G)$, and so  it can  be deduced from Theorem~\ref{t.brdet} that knowledge of $k(G)$ is not needed: $R(G)$ is determined entirely by information in the isotropic matroid and the quasi-tree $Q$.

\begin{remark}
The observation that the Bollob\'as--Riordan polynomial is delta-matroidal helps to explain the form of the deletion--contraction identity for the Bollob\'as--Riordan polynomial. More precisely it helps to explain why there is generally no known deletion--contraction identity when the edge being removed is a loop. The exponents of $x$ and $y$ depend on the rank function of the lower matroid. An orientable non-trivial loop $e$ of a ribbon graph $G$ is not a loop of $D(G)$ but is a loop of $D(G)_{\min}$. This means that $({D(G)/e})_{\min}$ is not generally the same as $(D(G)_{\min})/e$ and moreover
$(D(G)_{\min})/e$ cannot always be recovered from $({D(G)/e})_{\min}$.
\end{remark}

Most of the results on the Bollob\'as-Riordan polynomial in the literature (for example,~\cite{BR2,CP,Detal,EMM,EMM15,ES,KP}) hold not for the full four-variable polynomial but for the normalised two-variable version $x^{\gamma(G)/2}   R_G(x+1,y,1/\sqrt{xy},1)$.
This two-variable version of the polynomial has a particularly natural form when expressed in terms of delta-matroids. Define a  function $\sigma$ on delta-matroids by
$ \sigma(D) :=  \frac{1}{2}( r(D_{\max})  +   r(D_{\min}) )$,
and for $A\subseteq E(D)$,
$ \sigma_D(A) :=   \sigma(D|A)$, omitting the subscript $D$ whenever the context is clear.
We define the  \emph{two-variable Bollob\'as-Riordan polynomial} of a delta-matroid to be
\begin{equation}\label{e.tvbr}
\tilde{R}(D;x,y) :=  \sum_{A\subseteq E}  (x-1)^{\sigma(E)-\sigma(A)}(y-1)^{|A|-\sigma(A)}.
\end{equation}
One immediately notices from \eqref{e.tvbr} that if $D$ is a matroid with rank function $r$, then $\sigma(A) = r(A)$, so  $ \tilde{R}(D;x,y)$ is exactly the Tutte polynomial $T(D;x,y)$. It is also readily verified, using Proposition~\ref{p.rasp}\eqref{p.rasp.1}, that $\tilde{R}(D;x+1,y+1)  =    x^{w(D)/2}   R(D;x+1,y,1/\sqrt{xy},1) $.

It is well-known that the Tutte polynomial of a graph or matroid has a recursive deletion-contraction definition that expresses $T(M)$ as a $\mathbb{Z}[x,y]$-linear combination of Tutte polynomials. Analogously, the two-variable Bollob\'as-Riordan polynomial was shown to have a recursive deletion-contraction definition in~\cite{CMNR}, given in terms of  $R(D;x+1,y,1/\sqrt{xy},1)$, and  in~\cite{KMT}, given in terms of  $\tilde{R}(D;x,y)$. The difference in the two forms is due to the factor $x^{w(E)/2}$. Moreover, Krajewski, Moffatt, and Tanasa showed in~\cite{KMT} that $\tilde{R}(D;x,y)$ is the graph polynomial canonically associated with a natural Hopf algebra generated by delta-matroid deletion and contraction, just as the Tutte polynomial is the polynomial canonically  associated with a Hopf algebra generated by matroid deletion and contraction. Furthermore, $\tilde{R}(D)$ encodes fundamental combinatorial information about $D$.
\begin{theorem}\label{t.evals}
For any delta-matroid $D$, the following hold.
\begin{enumerate}
\item \label{t.evals1} $\tilde R(D;u/v+1,uv+1)$  gives the bivariate generating function of $D$ with respect to number of feasible sets of each size and rank:
\[  v^{\sigma(D)} u^{-w(D)/2}  \tilde R(D;u/v+1,uv+1)  = \sum_{A\subseteq E(D)} v^{|A|} u^{|E(D)|-\rho_D(A)};\]
\item \label{t.evals2} $\tilde R(D^*;x,y)= \tilde R(D;y,x)$;
\item \label{t.evals3} $\tilde{R}(D;1,1)=0$ unless $D$ is a matroid, in which case it equals the number of bases of $D$;

\item \label{t.evals4} $\tilde{R}(D;1,2)$ is the number of independent sets in $D_{\min}$;
\item \label{t.evals5}$\tilde{R}(D;2,1)$ is the number of spanning sets in $D_{\max}$;
\item \label{t.evals6}$\tilde{R}(D;2,2) = 2^{|E(D)|}$.
\end{enumerate}
\end{theorem}
\begin{proof}
Part \eqref{t.evals1} follows easily from the definition of $\tilde{R}(D)$  and Proposition~\ref{p.rasp}.

Let $E=E(D)$. Then~\eqref{t.evals2} follows by applying Proposition~\ref{p.rasp} to show that for any subset $A$ of $E$, the difference $\sigma_{D^*}(E)-\sigma_{D^*}(A) =|E-A| - \sigma_{D}(E-A)$.

For \eqref{t.evals3}, $\tilde{R}(D;1,1) =  \sum_{A\subseteq E}  0^{\sigma(E)-\sigma(A)}0^{|A|-\sigma(A)}$.
A term in the sum is non-zero if and only if $\sigma(E)-\sigma(A)=|A|-\sigma(A)=0$.
We have $\sigma(E)=\sigma(A)$ if and only if $r(D_{\max}) - r((D|A)_{\max}) + r(D_{\min})- r((D|A)_{\min})=0$, which occurs if and only if
$r(D_{\max}) = r((D|A)_{\max})$ and  $r(D_{\min}) = r((D|A)_{\min})$. On the other hand, $|A|-\sigma(A)=0$ if and only if $r((D|A)_{\max})= r((D|A)_{\min})=|A|$.

Therefore $\sigma(E)-\sigma(A)=|A|-\sigma(A)=0$ if and only if \[r(D_{\max}) = r((D|A)_{\max}) = r(D_{\min}) = r((D|A)_{\min})= |A|,\] which occurs if and only if $D$ is a matroid and $A$ is a basis of $D$.

For \eqref{t.evals4}, $\tilde{R}(D;2,1) =  \sum_{A\subseteq E}  0^{|A|-\sigma(A)}$. It follows from above that a term in the sum is non-zero
if and only if $r((D|A)_{\max})= r((D|A)_{\min})=|A|$. If $r((D|A)_{\min})=|A|$ then, by Proposition~\ref{p.rasp}, $r_{D_{\min}}(A) = |A|$. Consequently $A$ is independent in $D_{\min}$. On the other hand, if $A$ is independent in $D_{\min}$, then $r((D|A)_{\min})=|A|$, the only feasible set of $D|A$ is $A$, so $r((D|A)_{\max})= r((D|A)_{\min})=|A|$.

Recall that a spanning set $A$ of a matroid $M$, is a subset of $E(M)$ such that $r(A)=r(M)$.
Part~\eqref{t.evals5} follows from Parts~\eqref{t.evals2} and~\eqref{t.evals4}, because the complement of an independent set of a matroid is a spanning set of its dual.

Part~\eqref{t.evals6} is obvious.
\end{proof}

The final polynomial we consider in this section is the Krushkal polynomial of~\cite{Kr11}. This polynomial generalizes the Bollob\'as-Riordan polynomial by adding a parameter that records some information about the geometric dual. Although the Krushkal polynomial is also defined for  non-cellularly embedded graphs, here we restrict to cellularly embedded graphs, or, equivalently, ribbon graphs.
The \emph{Krushkal polynomial} of $G$, denoted by $K(G;x,y,a,b) \in \mathbb{Z}[x,y, a,b]$, is defined by
\begin{equation}
 K(G;x,y,a,b) := \sum_{A \subseteq E( G)}   (x - 1)^{r_G( E ) - r_G( A )}   y^{r_{G^*}(E)-r_{G^*}(A^c) } a^{\gamma_G(A)} b^{\gamma_{G^*}(A^c)}  .
  \end{equation}

We note that the exponent of $a$ is usually written as $k(A) - f(A) + n(A)$, which is equal to $\gamma(A)$ by Euler's formula, and similarly for the $b$ exponent. (An analogous comment holds for the $z$ exponent of the Bollob\'as-Riordan polynomial.) Also note, for comparison with the literature, that the exponents of $a$ and $b$ here are given by the Euler genus, rather than one-half of the Euler genus as in~\cite{Kr11}.

We showed that  $\gamma(A) = \spn_D(A)$ in Proposition~\ref{p.4b}\eqref{p.4b.3}. Using Corollary~\ref{c.gedu}, we have
$\gamma_{G^*}(A^c)=\gamma(G^*\ba A)= \spn(D(G^*\ba A))= \spn(D(G^*)\ba A)= \spn(D(G)^*\ba A) = \spn_{D(G)^*}(A^c)$.
{\sloppy \begin{definition}
Let $D=(E,\mathcal{F})$ be a delta-matroid. Then the \emph{Krushkal polynomial} $K(D;x,y,a,b) \in \mathbb{Z}[x,y,a,b]$, of $D$ is
                  \begin{align*}  \PullBack{K(D;x,y,a,b) }\\ &:=\sum_{A \subseteq E}   (x - 1)^{r_{D_{\min}}( E ) - r_{D_{\min}}( A )}   y^{r_{(D^*)_{\min}}( E ) - r_{(D^*)_{\min}}( A^c )} a^{ \spn_D(A)}  b^{ \spn_{D^*}(A^c)}. \end{align*}
\end{definition}}

We immediately have that the Krushkal polynomial of a ribbon graph is delta-matroidal.
\begin{theorem}
Let $G$ be a  ribbon graph. Then \[K(G;x,y,a,b)= K(D(G);x,y,a,b).\]
\end{theorem}

Krushkal observed in~\cite{Kr11} that, when $G$ is a plane graph, $T(G;x,y)=K(G;x,y-1,a,b)$. The analogous result holds for delta-matroids.
Using Equation~\eqref{eq:matrank}, for any matroid $M=(E,\mathcal B)$ and subset $A$ of $E$, we have
\begin{equation}
\label{eq:matnulldual}
 |A| - r_M(A) = r(M^*) - r_{M^*}(E-A).
\end{equation}
When $D$ is a matroid, this equation together with the fact that  $\spn_D(A)=\spn_{D^*}(A^c)=0$ implies that $T(D;x,y)=K(D;x,y-1,a,b)$.

For non-plane graphs, the Tutte polynomial can still be recovered from the Krushkal polynomial using the identity \[T(G;x,y+1) =y^{\gamma(G)/2} K(G;x,y,y^{1/2},y^{-1/2})\] (see~\cite{Bu12,Kr11}).
The Krushkal polynomial, however, contains not only the Tutte polynomial as a specialization, but also the Bollob\'as-Riordan polynomial at $w=1$ (see~\cite{Kr11}), and the Las~Vergnas polynomial (see~\cite{ACEMS11,Bu12}). Each of these results holds in the delta-matroid setting.
\begin{theorem}\label{t.spc}
Let $D$ be a delta-matroid. Then
\begin{enumerate}
\item $T(D;x,y+1) =y^{\spn(D)/2} K(D;x,y,y^{1/2},y^{-1/2})$;
\item $L(D;x,y,z)=  z^{\spn(D)/2} K(D; x,y-1,z^{-1/2},z^{1/2})$;
\item $R(D;x,y,z,1) =y^{\spn(D)/2} K(D;x,y,zy^{1/2},y^{-1/2})$.
\end{enumerate}

\end{theorem}
\begin{proof}
The first item follows from the third item upon noting that $T(D;x,y+1)=R(D;x,y,1,1)$.

For the second item, the exponents of $x$ in each summand on the left-hand and right-hand side agree.
Using Equation~\eqref{eq:matnulldual} and $(D_{\max})^*=(D^*)_{\min}$, we see that the exponents of $y-1$ in each summand on both sides agree.
For the $z$ term, the $z$ exponent of  each summand of $z^{\spn(D)/2} K(D; x,y-1,z^{-1/2},z^{1/2})$ is $\frac{1}{2}( \spn(D)-\spn_D(A)+\spn_{D^*}(A^c) )$.  By  Proposition~\ref{p.rasp},
\begin{align}
\PullBack{ \spn(D)-\spn_D(A)+\spn_{D^*}(A^c)}\nonumber\\
\label{eq.t.spc1} &=   r_{D_{\max}}( E ) - r_{D_{\min}}(E)
 -\rho_D (A) +r_{D_{\min}}(A)+ |E|-r_{D_{\min}}(E)-|A|  \\
&  \phantom{=} {} +r_{D_{\min}}(A)+\rho_{D^*} (A^c) -r_{D^*_{\min}}(A^c)- |E|+r_{D^*_{\min}}(E)+|A^c|- r_{D^*_{\min}}(A^c).\nonumber
\end{align}
By Lemma~\ref{le:rankdual}, $\rho_D (A)=\rho_{D^*} (A^c) $.  Additionally using Equation~\eqref{eq:matrank} and $(D_{\max})^*=(D^*)_{\min}$ we obtain
\[r_{D^*_{\min}}(A)  =  r_{(D_{\max})^*}(A)  =   |A|+r_{D_{\max}}(A^c) -r_{D_{\max}}(E). \]
Substituting these into Equation~\eqref{eq.t.spc1} allows us to rewrite it as
\begin{align*}
\PullBack{ \spn(D)-\spn_D(A)+\spn_{D^*}(A^c)}\\
&=      r_{D_{\max}}( E ) - 2r_{D_{\min}}(E)
  +2r_{D_{\min}}(A)
  -|A| \\
&\phantom{=} {}     -2 |A^c|-2r_{D_{\max}}(A) +2r_{D_{\max}}(E)
    + |E|-r_{D_{\max}}(E)
    +|A^c|\\
&=2(r_{D_{\max}}(E)-r_{D_{\min}}(E)    - r_{D_{\max}}(A)+ r_{D_{\min}}(A)).
\end{align*}
But this is just twice the $z$ exponent of the summands of $L(D;x,y,z)$.

For the third item, it is easy to see that the exponents of $x$ and $z$ on the left-hand and right-hand sides agree.
The exponent of $y$ on the right-hand side is
\[ \frac 12 (w(D) +2r_{(D^*)_{\min}}(E)-2r_{(D^*)_{\min}}(A^c) +w_D(A) - w_{D^*}(A^c)).\]
Using Proposition~\ref{p.rasp} and $(D_{\max})^*=(D^*)_{\min}$, it is  straightforward to show that this is equal to $n_{D_{\min}}(A)$, as required.
\end{proof}
The ribbon graph versions of Theorem~\ref{t.spc} from~\cite{ACEMS11,Bu12,Kr11} can be recovered by taking $D$ to be $D(G)$.

Since the Tutte polynomial can be defined in terms of matroid rank functions, it is interesting to observe that, by Proposition~\ref{p.rasp}, we can express $K(D)$, and therefore $R(D)$, entirely in terms of rank functions associated with $D$.
Let $E=E(D)$.
For $A\subseteq E(D)$, let
\begin{itemize}
\item $K_x(D,A)=r_{D_{\min}}( E ) - r_{D_{\min}}( A )$;
\item $K_y(D,A)=r_{(D^*)_{\min}}( E ) - r_{(D^*)_{\min}}( A^c )$;
\item $K_a(D,A)=\rho_D (A) -r_{D_{\min}}(A)- n_{D_{\min}}(E)+n_{D_{\min}}(A)$;
\item $K_b(D,A)=\rho_{D^*} (A^c) -r_{(D^*)_{\min}}(A^c)- n_{(D^*)_{\min}}(E)+n_{(D^*)_{\min}}(A^c)$.
\end{itemize}
Then
\begin{equation}
\label{kintermsofrank}
K(D;x,y,a,b)=\sum _{A\subseteq E} (x-1)^{K_x(D,A)}y^{K_y(D,A)}a^{K_a(D,A)}b^{K_b(D,A)}.
\end{equation}

The Tutte polynomial of a plane graph satisfies the duality relation: \[T(G; x,y)=T(G^*; y,x).\] This identity is actually matroidal as  $T(M; x,y)=T(M^*; y,x)$, and the result for graphs follows since $M(G^*)=M(G)^*$ when $G$ is a plane graph. Similar duality identities were shown for the Bollob\'as-Riordan polynomial in~\cite{ES11,Mo1}, the Krushkal polynomial in~\cite{Kr11}, and the Las~Vergnas polynomial in~\cite{Las78}. The following theorem shows that each of these duality relations holds on the level of delta-matroids.
\begin{theorem}
\label{KRL}
Let $D$ be a delta-matroid.  Then
\begin{enumerate}
\item $K(D; x,y-1,a,b)=K(D^*; y,x-1,b,a)$;
\item $x^{\spn(D)/2} R(D; x+1, y, 1/\sqrt{xy},1) = y^{\spn(D^*)/2} R(D^*; y+1, x, 1/\sqrt{xy},1)$;
\item $L(D;x,y,z) =  z^{\spn(D)} L(D^*;y,x,z^{-1})$.
\end{enumerate}

\end{theorem}
\begin{proof}
The first part can be proven by writing down the sums for the two sides of the equation and observing that summing over all $A\subseteq E$ is the same as summing over all $A^c\subseteq E$. The second and third parts then follow by Theorem~\ref{t.spc}.
\end{proof}
The corresponding duality relations for the ribbon graph versions of the polynomials from~\cite{ES11,Kr11,Las78,Mo1} follow  from the theorem as $D(G^*)=D(G)^*$.

\medskip

We conclude  with an application to knot theory. There is a well-known way to associate a plane graph $G_L$ to an alternating link diagram $L$ such that the  Kauffman bracket $\langle L \rangle$, or Jones polynomial (if the writhe of the link is known), of $L$ can be recovered from the Tutte polynomial of $G_L$ together with knowledge of $k(G_L)$ (see~\cite{This}).
Recently, Dasbach,  Futer, Kalfagianni, Lin and Stoltzfus, in~\cite{Detal}, extended this result to all link diagrams (including those that are not alternating) by describing how a ribbon graph $\mathbb{A}_L$ can be associated with any link diagram $L$. It was also shown in~\cite{Detal} that the Kauffman bracket  and Jones polynomial (again provided the writhe of the link is known) of $L$ can be recovered from the Bollob\'as-Riordan polynomial of $\mathbb{A}_L$ together with knowledge of $k(\mathbb{A}_L)$.

If a link diagram is split, then we can use a sequence of Reidemeister II moves to obtain an equivalent non-split diagram. If we construct  $\mathbb{A}_L$ from this diagram then we know it is a connected ribbon graph, so we no longer need knowledge of $k(\mathbb{A}_L)$.

Recalling that the Tutte polynomial of $G_L$  can be recovered from its cycle matroid $M(G_L)$, this means that the Kauffman bracket of an {\em alternating} link is matroidal in the sense that it can be recovered from a matroid associated with any of its non-split diagrams.  Since the Bollob\'as-Riordan polynomial of $\mathbb{A}_K$ is determined by $D(\mathbb{A}_K)$, this means that in general, the Kauffman bracket can be regarded as a delta-matroidal object.
\begin{theorem}
The Kauffman bracket  of a link is delta-matroidal, in the sense that it is  determined by  delta-matroids associated with non-split link diagrams.
\end{theorem}
This result also holds for virtual link diagrams by~\cite{CV} and for links in real projective space by~\cite{most}, and, if we know the writhe of the link diagram we can extend both results to the Jones polynomial. Finally, using~\cite{Mo1}, we can show the homfly-pt polynomial of a class of links is delta-matroidal.

\section*{Acknowledgements}
We thank the anonymous referees for a very careful reading. Their comments have significantly improved the presentation of this paper. We also thank Tony Nixon, Sang-il Oum, and Lorenzo Traldi for many helpful suggestions and comments.

\bibliographystyle{amsplain}

\end{document}